\documentclass[onefignum,onetabnum]{siamart171218}

\newcommand{\jatnote}[1]{\textcolor{blue}{\textbf{[JAT: #1]}}}
\renewcommand{\jatnote}[1]{}

\usepackage{thmtools}
\usepackage{thm-restate}
\usepackage{hyperref}
\usepackage{xspace}
\usepackage{amssymb}
\usepackage{pifont}
\usepackage{accents}
\newcommand{\ubar}[1]{\underaccent{\bar}{#1}}

\newcommand{\defn}{:\,=}
\newcommand{\BEAS}{\begin{eqnarray*}}
\newcommand{\EEAS}{\end{eqnarray*}}
\newcommand{\BEA}{\begin{eqnarray}}
\newcommand{\EEA}{\end{eqnarray}}
\newcommand{\BEQ}{\begin{equation}}
\newcommand{\EEQ}{\end{equation}}
\newcommand{\BIT}{\begin{itemize}}
\newcommand{\EIT}{\end{itemize}}
\newcommand{\BNUM}{\begin{enumerate}}
\newcommand{\ENUM}{\end{enumerate}}

\newcommand{\beas}{\begin{eqnarray*}}
\newcommand{\eeas}{\end{eqnarray*}}
\newcommand{\bea}{\begin{eqnarray}}
\newcommand{\eea}{\end{eqnarray}}
\newcommand{\beq}{\begin{equation}}
\newcommand{\eeq}{\end{equation}}
\newcommand{\bit}{\begin{itemize}}
\newcommand{\eit}{\end{itemize}}
\newcommand{\ben}{\begin{enumerate}}
\newcommand{\een}{\end{enumerate}}

\newcommand{\ba}{\begin{array}}
\newcommand{\ea}{\end{array}}
\newcommand{\bbm}{\begin{bmatrix}}
\newcommand{\ebm}{\end{bmatrix}}

\newcommand{\eg}{e.g., }
\newcommand{\ie}{i.e., }

\newcommand{\ones}{\mathbf 1}

\newcommand{\reals}{{\mbox{\bf R}}}

\newcommand{\sym}{{\mbox{\bf S}}}  

\newcommand{\range}{\mathop{\bf range}}
\newcommand{\rank}{\mathop{\bf rank}}
\newcommand{\nullspace}{{\mathop {\bf null}}}
\newcommand{\tr}{\mathop{\bf tr}}

\newcommand{\diag}{\mathop{\bf diag}}

\newcommand{\twonorm}[1]{\left\|#1\right\|_2}
\newcommand{\norm}[1]{\left\|#1\right\|}
\newcommand{\fronorm}[1]{\left\|#1\right\|_{\mbox{\tiny{F}}}}
\newcommand{\opnorm}[1]{\left\|#1\right\|_{\mbox{\tiny{\textup{op}}}}}
\newcommand{\nucnorm}[1]{\left\|#1\right\|_*}

\newcommand{\argmin}{\mathop{\rm argmin}}

\newcommand{\conv}{\mathop{\bf conv}}

\newcommand{\constantA}{\frac{\delta_S}{\sigma_{\min}(\mathcal{A}_V)} +\opnorm{S}}
\newcommand{\constantD}{\frac{\constantE}{1+\kappa_V}}
\newcommand{\constantE}{\phi}
\newcommand{\constantB}{\frac{1}{4}\biggr[ \sqrt{2\constantE^2+4\biggr(\constantA\biggr) +4\constantD}+\sqrt{2}\constantE\biggr]^2}

\newcommand{\bigO}[1]{\mathcal{O}(#1)}

\newcommand{\pval}{p_\star}
\newcommand{\dval}{d_\star}
\newcommand{\dm}{n}
\newcommand{\cons}{m}
\newcommand{\rsol}{r_\star}
\newcommand{\xsol}{X_\star}
\newcommand{\ysol}{y_\star}
\newcommand{\zsol}{Z(\ysol)}

\newcommand{\ssol}{S_\star}
\newcommand{\xsolset}{\mathcal{X}_\star}

\newcommand{\vspacex}{\mathcal{V}_\star}
\newcommand{\uspacex}{\mathcal{U}_\star}

\newcommand{\vrepox}{{V_\star}}
\newcommand{\urepox}{{U_\star}}

\newcommand{\infeasx}{X_{\textup{infeas}}}
\newcommand{\objx}{X_{\text{obj}}}
\newcommand{\objp}{g_\alpha}
\newcommand{\objc}{C}
\newcommand{\Amap}{\mathcal{A}}
\newcommand{\bvec}{b}

\newcommand{\Av}[1]{\Amap_{#1}}

\newcommand{\ineqov}[1]{\overset{#1}{\leq}}
\newcommand{\alg}{\mathcal{G}}

\newcommand{\trux}{\bar{X}}

\newcommand{\minfeas}{MinFeasSDP\xspace}
\newcommand{\minobj}{MinObjSDP\xspace}

\usepackage{lipsum}
\usepackage{amsfonts}
\usepackage{graphicx}
\usepackage{epstopdf}
\usepackage{algorithmic}
\ifpdf
  \DeclareGraphicsExtensions{.eps,.pdf,.png,.jpg}
\else
  \DeclareGraphicsExtensions{.eps}
\fi

\newsiamremark{remark}{Remark}
\newsiamremark{hypothesis}{Hypothesis}
\crefname{hypothesis}{Hypothesis}{Hypotheses}
\newsiamthm{claim}{Claim}

\headers{Solving SDPs via Approx. Complementarity}{}

\title{An Optimal-Storage Approach to Semidefinite Programming using Approximate Complementarity}

\author{Lijun Ding\thanks{Operations Research and Information Engineering, Cornell University, Ithaca, NY, USA
  (\email{ld446@cornell.edu}, \email{udell@cornell.edu}).}
\and Alp Yurtsever\thanks{Laboratory for Information and Inference Systems (LIONS), EPFL, \'{E}cublens, Vaud, Switzerland
  (\email{alp.yurtsever@epfl.ch}, \email{volkan.cevher@epfl.ch}).}
\and Volkan Cevher\footnotemark[3] \and \linebreak
	Joel A.~Tropp\thanks{Computing and Mathematical Sciences, California Institute of Technology, Pasadena,
	CA, USA (\email{jtropp@cms.caltech.edu}).}
\and Madeleine Udell\footnotemark[2]}

\usepackage{amsopn}

\ifpdf
\hypersetup{
	pdftitle={An Optimal-Storage Approach to Semidefinite Programming using Approximate Complementarity},
	pdfauthor={Lijun Ding, Alp Yurtsever, Volkan Cevher, Joel A. Tropp and
		Madeleine Udell}
}
\fi
\usepackage{subcaption}
\usepackage{graphicx}

\begin{document}

	\maketitle

	\begin{abstract}
		This paper develops a new storage-optimal algorithm that provably
		solves almost all semidefinite programs (SDPs).
%
%
%
%
		This method is particularly effective for weakly constrained SDPs.
		The key idea is to formulate an approximate complementarity principle:
		Given an approximate solution to the dual SDP,
		the primal SDP has an approximate solution whose range is
		contained in the eigenspace with small eigenvalues of the dual slack matrix.
		For weakly constrained SDPs, this eigenspace has very low dimension,
		so this observation significantly reduces the search space for the primal solution.
		This result suggests an algorithmic strategy that can be implemented with minimal storage:
		(1) Solve the dual SDP approximately;
		(2) compress the primal SDP to the eigenspace with small eigenvalues of the dual slack matrix;
		(3) solve the compressed primal SDP.
		The paper also provides numerical experiments showing that this approach is successful
		for a range of interesting large-scale SDPs.
%
%
	\end{abstract}

	\begin{keywords}
		Semidefinite programs, Storage-optimality, Low rank, Complementary slackness,
		Primal recovery.
	\end{keywords}

	\begin{AMS}
		90C06, 90C2, 49M05
	\end{AMS}
	\section{Introduction}\label{sec: introduction}

	Consider a semidefinite program (SDP) in the standard form
	\beq\label{p}\tag{P}
	\ba{ll}
	\mbox{minimize} & \tr(CX)\\
	\mbox{subject to} & \mathcal{A}X =  b
	\quad\text{and}\quad
	X \succeq 0. \\
	\ea
	\eeq
	The primal variable is the symmetric, positive-semidefinite matrix
	$X \in \sym^\dm_+$. The problem data comprises
	a symmetric (but possibly indefinite) objective matrix $C\in \sym^\dm$,
	a righthand side $b\in\reals^m$, and
	a linear map $\mathcal{A}: \reals^{\dm\times \dm} \rightarrow \reals^m$
	with rank $m$, which can be expressed explicitly as
	$[\mathcal{A}H]_i = \tr(A_iH),i=1,\dots,m$ for some symmetric $A_i\in \sym^\dm$ and any $H\in \reals^{\dm \times \dm}$.
	The notation $\tr(\cdot)$ stands for
	the trace operation: $\tr(A) = \sum_{i=1}^n a_{ii}$ for any $A\in \reals^{\dm \times \dm}$ with $(i,j)$-th entry
	$a_{ij}\in \reals$.

	SDPs form a class of convex optimization problems
	with remarkable modeling power.  But SDPs are challenging to solve because they involve
	a matrix variable $X \in \sym_+^\dm \subset \reals^{n \times n}$ whose dimension $n$
	can rise into the millions or billions.
	For example, when using a matrix completion SDP in a recommender system,
	$n$ is the number of users and products;
	when using a phase retrieval SDP to visualize a biological sample,
	$n$ is the number of pixels in the recovered image.
	In these applications, most algorithms are prohibitively expensive
	because their storage costs are quadratic in $n$.

	How much memory should be required to solve this problem?
	Any algorithm must be able to query the problem data and
	to report a representation of the solution.
	Informally, we say that an algorithm uses \emph{optimal storage} if the working
	storage is no more than a constant multiple of the storage required for these
	operations~\cite{yurtsever2017sketchy}.
	(See \Cref{sec:storage-opt} for a formal definition.)

	It is not obvious how to develop storage-optimal SDP algorithms.
	To see why, recall that all \emph{weakly-constrained} SDPs ($m = \bigO{\dm}$)
	admit low-rank solutions~\cite{barvinok1995problems,pataki1998rank},
	which can be expressed compactly in factored form.  For these problems,
	a storage-optimal algorithm cannot even instantiate the matrix variable!
	One natural idea is to introduce
	an explicit low rank factorization of the primal variable $X$
	and to minimize the problem over the factors \cite{burer2003nonlinear}.

	Methods built from this idea provably work
	when the size of the factors is sufficiently large \cite{boumal2016non}.
	However, recent work~\cite{waldspurger2018rank} shows that they cannot provably solve
	all SDPs with optimal storage; see \Cref{sec:related-work}.

	In contrast, this paper develops a new algorithm that provably solves all
	\emph{regular} SDPs, \ie SDPs with
	strong duality, unique the primal and dual solutions,
	and strict complementarity.
	These standard conditions not only hold generically \cite[Definition 19]{alizadeh1997complementarity},
	but also in many applications \cite{LM2020regular}. We defer
	the detailed description of these conditions in \Cref{sec:assumptions}.

Our method begins with the Lagrange dual of the primal SDP \cref{p}, 
	\beq \label{d}\tag{D}
	\ba{ll}
	\mbox  {maximize} & b^\top y \\
	\mbox{subject to}  & C- \mathcal{A}^\top y\succeq 0 \\
	\ea
	\eeq
	 with dual variable $y \in \reals^m$.
	The vector $b^\top$ is the transpose of $b$,
	and the linear map $\mathcal{A}^\top: \reals^m \rightarrow \reals^{\dm\times \dm}$
	is the adjoint of the linear map $\mathcal{A}$. Note the range of $\mathcal{A}^\top$ is in $\sym^\dm$
	because $C$ and $A_i$s are symmetric.
%
%
It is straightforward to compute an approximate solution
to the dual SDP \cref{d} with optimal storage using methods described in \Cref{sec: sd}.
The challenge is to recover a primal solution
from the approximate dual solution.

To meet this challenge, we develop a new
\emph{approximate complementarity principle}
that holds for regular SDP:
Given an approximate dual solution $y$, we prove that
there is
a primal approximate solution $X$ whose range is contained in
the eigenspace with small eigenvalues of the dual slack matrix $C - \mathcal{A}^\top y$.
This principle suggests an algorithm:
we solve the primal SDP by searching over matrices with the
appropriate range.
This recovery problem is a (much smaller) SDP
that can be solved with optimal storage.

%

	\subsection{Regularity Assumptions}
	\label{sec:assumptions}



	First, assume that the primal~\cref{p} has a solution, say, $\xsol$
	and the dual~\cref{d} has a \emph{unique} solution $\ysol$.
	We require that strong duality holds:\vspace{-2pt}
	\begin{equation} \label{eqn:strong-duality}
	p_{\star} := \tr(C\xsol) = b^\top \ysol =: d_{\star}.
	\end{equation}
	The condition~\cref{eqn:strong-duality} follows, for example, from Slater's constraint qualification.

	Strong duality and feasibility imply that the solution $\xsol$
	and the dual slack matrix $C - \mathcal{A}^\top\ysol$
	satisfy the complementary slackness condition:
\vspace{-2pt}
	\begin{equation} \label{eqn:complementary-slackness}
	\xsol (C - \mathcal{A}^\top \ysol) = 0.
	\end{equation}
which implies that
\[
\rank(\xsol) + \rank(C - \mathcal{A}^\top\ysol) \leq  n.
\]
%
	To ensure that we are not in a degenerate situation, we make the stronger assumption
	that every solution pair $(\xsol, \ysol)$ satisfies
	the stronger \emph{strict complementarity} condition:\vspace{-2pt}
	\begin{equation} \label{eqn:strict-complement}
	\rank(\xsol) + \rank(C - \mathcal{A}^\top\ysol) = n.
	\end{equation}
	%
	Note that these assumptions ensure that \emph{all} solutions have the same rank,
	and therefore that the primal solution is actually \emph{unique} \cite[Corollary 2.5]{lemon2016low}.
	In particular, the rank $\rsol$ of the solution $\xsol$ satisfies
	the Barvinok--Pataki bound
	$\binom{ \rsol + 1 }{ 2 } \leq m$.

	To summarize, all results in this paper hold under the regularity assumptions:
	primal attainability, dual uniqueness, strong duality, and strict complementarity.
	These conditions hold generically conditioning on primal and dual attainability;
	\ie for every SDP satisfying primal and dual attainability
	outside of a set of measure $0$~\cite{alizadeh1997complementarity}. The conditions
	or a large fraction of them are also satisfied for non-generic SDP in applications \cite{LM2020regular}.

	\subsection{Optimal Storage}\label{sec:storage-opt}

	Following~\cite{yurtsever2017sketchy},
	let us quantify the storage necessary
	to solve every SDP~\cref{p}
	that satisfies our assumptions in \Cref{sec:assumptions}
	and that admits a solution with rank $\rsol$.
%

	First, it is easy to see that $\Theta(n\rsol)$ numbers are sufficient to represent
	the rank-$\rsol$ solution in factored form.  This cost is also necessary because
	\emph{every} rank-$\rsol$ matrix is the solution to some SDP from our
	problem class.


	To hide the internal complexity of the optimization problem~\cref{p},
	we will interact with the problem data using \emph{data access oracles}.
	Suppose we can perform any of the following operations
	on arbitrary vectors $u,v\in \reals^n$ and $y \in \reals^m$:\vspace{-2pt}
	\begin{equation}
		\begin{aligned}\label{eq: efficientOperation}
			u \mapsto Cu
			\quad\text{and}\quad
			(u,v) \mapsto \Amap (uv^\top)
			\quad\text{and}\quad
			(u, y) \mapsto (\mathcal A^\top y)u.
		\end{aligned}
	\end{equation}
	These oracles enjoy simple implementations in many concrete applications.
	The input and output of these operations clearly involve storing $\Theta(m+n)$ numbers.


	In summary, any method that uses these data access oracles to solve
	every SDP from our class must store $\Omega(m + n\rsol)$ numbers.
	We say a method has \emph{optimal storage} if the working storage
	provably achieves this bound.

	For many interesting problems, the number $m$ of constraints is
	proportional to the dimension $n$.  Moreover, the rank $\rsol$ of
	the solution is constant or logarithmic in $n$.  In this case,
	a storage-optimal algorithm has working storage $\tilde{\mathcal{O}}(n)$,
	where the tilde suppresses log-like factors.


\begin{remark}[Applications]
	The algorithmic framework we propose is most useful
	when the problem data has an efficient
	representation and the three operations in~\cref{eq: efficientOperation}
	can be implemented with low arithmetic cost.
	For example, it is often the case that the matrix $C$ and the linear map $\mathcal{A}$
	are sparse or structured. 
	This situation occurs in the maxcut relaxation~\cite{goemans1995improved},
	matrix completion~\cite{srebro2005rank}, 
	phase retrieval~\cite{chai2010array,waldspurger2015phase}, 
	and community detection~\cite{mathieu2010correlation}. 
	See \cite{recht2010guaranteed} for some other examples. We  expect that the assumptions mentioned in this paper (or most of them)  to be satisfied for these problems. The cases of matrix completion, and community detection have been verified in \cite{LM2020regular}.
\end{remark}

	\subsection{From Strict Complementarity to Storage Optimality} \label{sec:Main Idea}


	Suppose that we have computed the \emph{exact} unique dual solution $\ysol$.
	Complementary slackness~\cref{eqn:complementary-slackness}
	and strict complementarity~\cref{eqn:strict-complement}
	ensure that
	$$
	\range(\xsol) \subset \nullspace(C - \mathcal{A}^\top \ysol)
	\quad\text{and}\quad
	\dim( \nullspace(C - \mathcal{A}^\top \ysol)) = \rank(\xsol).
	$$
	Therefore, the slack matrix identifies the range of the primal solution.

	Let $\rsol$ be the rank of the primal solution.
	Construct an orthonormal matrix $\vrepox \in \reals^{n \times \rsol}$
	whose columns span $\nullspace(C - \mathcal{A}^\top\ysol)$.
	The compression of the primal problem~\cref{p}
	to this subspace is
	\beq\label{pp}
	\ba{ll}
	\mbox{minimize} & \tr( C \vrepox S \vrepox^\top )\\
	\mbox{subject to} & \mathcal{A}(\vrepox S \vrepox^\top) =  b \quad\text{and}\quad
	S\succeq 0.
	\ea
	\eeq
	The variable $S\in \sym^{\rsol}_+$ is a low-dimensional matrix when $\rsol$ is small.
	If $S_{\star}$ is a solution to~\cref{pp},
	then $\xsol = \vrepox S_{\star} \vrepox^\top$ is a solution
	to the original SDP~\cref{p}.

	This strategy for solving the primal SDP can be implemented
	with a storage-optimal algorithm.  Indeed, the variable $y$ in the
	dual SDP~\cref{d} has length $m$, so there is no obstacle to
	solving the dual with storage $\Theta(m + n)$
using the subgradient type method described in \Cref{sec: sd}.
	We can compute the subspace $\vrepox$ using the randomized range finder~\cite[Alg.~4.1]{HMT11:Finding-Structure}
	with storage cost $\Theta(n \rsol)$.
	Last, we can solve	the compressed primal SDP~\cref{pp} using working storage
	$\Theta( m+n + \rsol^2 )$ via the matrix-free method from~\cite{diamond2016matrix,ocpb:16}.
	The total storage is the optimal $\Theta( m + n \rsol )$.
	Furthermore, all of these algorithms can be implemented with the data
	access oracles~\cref{eq: efficientOperation}.

	Hence --- assuming exact solutions to the optimization problems ---
	we have developed a storage-optimal approach to the SDP~\cref{p},
	summarized in Table~\ref{tb: conceptualAndNumericalApproach}[left].

	\subsection{The Approximate Complementarity Principle}
	\label{sec:approx-comp-intro}

	A major challenge remains:
	one very rarely has access to an exact dual solution!  Rather, we
	usually have an approximate dual solution, obtained via some iterative dual
	solver.



	This observation motivates us to formulate a new \emph{approximate complementarity principle}.
	For now, assume that $\rsol$ is known.  Given an approximate dual solution $y$, we can
	construct an orthonormal matrix $V \in \reals^{n \times \rsol}$
	whose columns are eigenvectors of $C-\mathcal{A}^\top y$
	with the $\rsol$ smallest eigenvalues.  
	Roughly speaking, the primal problem~\cref{p} admits an \emph{approximate} solution $X$
	whose range is contained in $\range(V)$.  We show the approximate solution is close to the
	true solution as measured in terms of suboptimality, infeasibility,
	and distance to the solution set.


	We propose to recover the approximate primal solution by solving
	the semidefinite least-squares problem
	\beq
	\ba{ll}\label{minfeas}\tag{\minfeas}
	\mbox{minimize} & \frac{1}{2}\norm{\mathcal{A}(VSV^\top) -  b}^2\\
	\mbox{subject to} &  S\succeq 0
	\ea
	\eeq
	with variable $S\in \sym^{\rsol}_+$.  Given a solution $\hat{S}$
	to~\cref{minfeas}, we obtain an (infeasible) approximate  solution
	$\infeasx = V \hat{S} V^\top$ to the primal problem.

	In fact, it is essential to relax our attention to infeasible solutions
	because the feasible set of~\cref{p} should almost never
	contains a matrix with range $V$!
	This observation was very surprising to us,
	but it seems evident in retrospect.
	(For example, using a dimension-counting argument
	together with \cref{thm: uniqueness}.)
\begin{table}{\footnotesize
		\caption{Exact and Practical Primal Recovery}\label{tb:
			conceptualAndNumericalApproach}}
	\centering
	\begin{tabular}{|c|l|l|}
		\hline
		Step &		\textbf{Exact Primal Recovery}	 &  \textbf{Practical Primal Recovery} \\
		\hline
		$1$		&	Compute dual solution $ \ysol$ &  Compute approximate dual solution
		$y$\\
		\hline
		$2$		&	Compute basis $\vrepox$  & Compute $\rsol$ eigenvectors \\
		& for $\nullspace(C - \mathcal A^\top \ysol)$ & of $C - \mathcal A^\top y$ with smallest  eigenvalues; \\
		& &	collect as columns of matrix $V$\\
		\hline
		$3$	&		 Solve the compressed SDP \cref{pp} &  Solve \cref{minfeas}\\
		\hline
	\end{tabular}
\end{table}
	The resulting framework appears in \cref{tb: conceptualAndNumericalApproach}[right].
	This approach for solving \cref{p} leads to storage-optimal algorithms for the
	same reasons described in~\Cref{sec:Main Idea}.  Our first
	main result ensures that this technique results in a provably
	good solution to the primal SDP~\cref{p}.

	\begin{theorem*}[Main theorem, informal]\label{thm:informal mainthm}
		Instate the regularity assumptions of~\Cref{sec:assumptions}.
		Suppose we have found a dual vector $y$ with suboptimality
		$\epsilon := \dval - \bvec ^\top y \leq \mathrm{const}$.
		Consider the primal reconstruction $\infeasx$ obtained
		by solving~\cref{minfeas}. 
		Then we may bound the distance between $\infeasx$ to the primal solution $\xsol$ by
		\[ \fronorm{\infeasx-\xsol} = \mathcal{O}(\sqrt{\epsilon}). \]
		The constant in the $\mathcal O$ depends on the problem data $\Amap$,
		$\bvec$, and $\objc$.
	\end{theorem*}

	We state and prove the formal result as \cref{thm: minfeas}.
	As stated, this guarantee requires knowledge of the rank $r^\star$
	of the solution;
	in \Cref{sec: algorithmguarantee}, we
	obtain a similar guarantee using an estimate for $r^{\star}$.

	\subsection{Paper Organization}
	We discuss related work in \Cref{sec:related-work} with a focus on storage.
	\Cref{sec:basics} contains an overview of our notation and more detailed
	problem assumptions.
	\Cref{sec: oracleinequalities} uses the approximate complementarity principle to develop
	practical, robust, and theoretically justified compressed SDPs such
	as \cref{minfeas} for solving~\cref{p}.
	These compressed SDPs are accompanied by detailed bounds on the quality of the
	computed solutions as compared with the true solution.
	\Cref{sec: algorithmguarantee} contains practical suggestions in
	solving these compressed SDPs such as choosing parameters, and
	checking the solution quality numerically.
	Next, we turn to algorithms for solving the dual SDP:
	we explain how to compute an approximate dual solution efficiently in \Cref{sec: sd},
	which provides the last ingredient for a complete method to solve~\cref{p}.
	\Cref{sec: numerics} shows numerically that the method
	is effective in practice. We conclude the paper with a discussion on contributions and future research 
	directions 
	in \Cref{sec:conclusions}.

		\section{Related Work} \label{sec:related-work}

	Semidefinite programming can be traced to a 1963 paper of Bellman \& Fan~\cite{bellman1963systems}.
	Related questions emerged earlier in control theory,
	starting from Lyapunov's 1890 work on stability of dynamical systems.
	There are many classic applications in matrix analysis, dating to the 1930s.
	Graph theory provides another rich source of examples, beginning from the 1970s.
	See~\cite{boyd1994linear,vandenberghe1996semidefinite, todd2001semidefinite, boyd2004convex} for more history and problem formulations.

	\subsection{Interior-Point Methods}

	The first reliable algorithms for semidefinite programming were interior-point methods (IPMs).
	These techniques were introduced independently by Nesterov \& Nemirovski~\cite{nesterov1989self,nesterov1994interior}
	and Alizadeh~\cite{alizadeh1991combinatorial,alizadeh1995interior}. 

	The success of these SDP algorithms motivated new applications.
	In particular, Goemans \& Williamson~\cite{goemans1995improved} 
	used semidefinite programming to design an approximation algorithm
	to compute the maximum-weight cut in a graph.
	Early SDP solvers could only handle graphs with a few hundred vertices~\cite[Sec.~5]{goemans1995improved}
	although computational advances quickly led to IPMs that could solve
	problems with thousands of vertices~\cite{benson2000solving}.

	IPMs form a series of unconstrained problems whose solutions are feasible for
	the original SDP, and move towards the solutions
	of these unconstrained problems using Newton's method.
	As a result, IPMs converge to high accuracy in very few iterations,
	but require substantial work per iteration.
	To solve a standard-form SDP with an $n \times n$ matrix variable
	and with $m$ equality constraints,
	a typical IPM requires
	$\bigO{\sqrt{n}\log(\frac{1}{\epsilon})}$ iterations to reach a solution with accuracy $\epsilon$
	 (in terms of objective value) \cite{nesterov2013introductory},
  and $\bigO{mn^3 + m^2 n^2 + m^3}$ arithmetic operations
  per iteration (when no structure is concerned)\cite{alizadeh1998primal},
	so $\bigO{\sqrt{n}\log(\frac{1}{\epsilon})(mn^3 + m^2 n^2 + m^3)}$ arithmetic operations in total.
	Further, a typical IPM requires at least $\Theta(n^2 + m + m^2)$ memory
	 not including the storage of data representation (which
	takes $\Theta(n^2 m)$ memory if no structure is assumed)\cite{alizadeh1998primal}.

	As a consequence, these algorithms are not effective for solving large problem instances,
	unless they enjoy a lot of structure.
	Hence researchers began to search for methods that could scale to larger problems.

	\subsection{First-Order Methods}

	One counterreaction to the expense of IPMs was to develop first-order optimization algorithms for SDPs.
	This line of work began in the late 1990s,
	and it accelerated as SDPs emerged in the machine learning and signal processing literature in the 2000s.

	Early on, Helmberg \& Rendl~\cite{helmberg2000spectral} 
	proposed
	a spectral bundle method for solving an SDP in dual form,
	and they showed that it converges to a dual solution when the trace of $\xsol$ is constant.
	In contrast to IPMs, the spectral bundle method has low per iteration complexity.
  On the other hand, the convergence rate is not known, and there is no convergence guarantee
  on the primal side. so there is no explicit control on the storage and arithmetic costs.

	Popular first-order algorithms include
	the proximal gradient method~\cite{rockafellar1976monotone}, 
	accelerated variants~\cite{beck2009fast} 
  and the alternating direction method of multipliers~\cite{gabay1975dual,glowinski1975approximation,boyd2011distributed,scs}.
	These methods provably solve the original convex formulation of~\cref{p}.
	But they 
 all store the full primal matrix variable, so they are not storage-efficient.

	Recently, Friedlander \& Macedo~\cite{friedlander2016low} 
	have proposed
	a novel first-order method that is based on gauge duality, rather than Lagrangian duality.
	This approach converts an SDP into an eigenvalue optimization problem.
	The authors propose a mechanism for using a dual solution to construct a primal solution.
	This paper is similar in spirit to our approach,
	but it lacks an analysis of the accuracy of the primal solution.
	Moreover, it only applies to problems with a positive-definite objective, \ie$C\succ 0$.

	\subsection{Storage-Efficient First-Order Methods}

	Motivated by problems in signal processing and machine learning,
	a number of authors have revived the conditional gradient method (CGM)
	~\cite{frank1956algorithm,levitin1966constrained}.
	In particular, Hazan~\cite{hazan2008sparse} 
	suggested using CGM for semidefinite programming. Clarkson~\cite{clarkson2010coresets} 
	developed a new analysis,
	and Jaggi~\cite{jaggi2013revisiting}, showed how this algorithm applies to a wide range of interesting problems.

	The appeal of the CGM is that it computes an approximate solution to an SDP as a sum of rank-one updates;
	each rank-one update is obtained from an approximate eigenvector computation.
	In particular, after $t$ iterations, the iterate has rank at most $t$.
	This property has led to the exaggeration that CGM is a ``storage-efficient'' optimization method when
	terminated early enough.
	Unfortunately, CGM converges very slowly, so the iterates do not have controlled rank.
  The literature describes many heuristics for
  attempting to control the rank of the iterates~\cite{rao2013conditional,yurtsever2015scalable},
	but these methods all lack guarantees.

	Very recently, some of the authors of this paper~\cite{yurtsever2017sketchy} 
	have shown how to use CGM to design a storage-optimal algorithm for a class of semidefinite programs
	by sketching the decision variable.
	This algorithm does not apply to standard-form SDPs, and it inherits the slow convergence of CGM.
	Nevertheless, the sketching methodology holds promise as a way to design storage optimal solvers,
	particularly together with algorithms that generalize CGM and that do apply to standard-form SDPs
	\cite{yurtsever2018conditional, yurtsever2019conditional}.

	We also mention a subgradient method developed by Renegar~\cite{renegar2014efficient}
	that can be used to solve either the primal or dual SDP.
	Renegar's method has a computational profile similar to CGM,
	and it does not have controlled storage costs.

	\subsection{Factorization Methods}\label{sec: factorizationMethods}

	There is also a large class of heuristic SDP algorithms based on matrix factorization.
	The key idea is to factorize the matrix variable $X = FF^\top, F\in \reals^{n\times r}$
	and to reformulate the SDP~\cref{p} as
	\begin{equation} \label{eqn:burer-monteiro}
	\ba{ll}
	\mbox{minimize} & \tr( CFF^\top ) \\
	\mbox{subject to} & \mathcal{A}(FF^\top) = b.
	\ea
	\end{equation}
	We can apply a wide range of nonlinear programming methods to optimize~\cref{eqn:burer-monteiro}
	with respect to the variable $F$.
	In contrast to the convex methods described above,
	these techniques only offer incomplete guarantees on storage, arithmetic, and convergence.

	The factorization idea originates in the paper~\cite{homer1997design} 
	of Homer \& Peinado.
	They focused on the Max-Cut SDP, and the factor $F$ was a \emph{square} matrix, \ie $r=n$.
	These choices result in an unconstrained nonconvex optimization problem
	that can be tackled with a first-order optimization algorithm.

	Theoretical work of Barvinok~\cite{barvinok1995problems} and Pataki~\cite{pataki1998rank}
	demonstrates that the primal SDP~\cref{p}
	always admits a solution with rank $r$, with $\binom{r + 1}{2} \leq m$.
	(Note, however, that the SDP can have solutions with much lower or higher rank.)

	Inspired by the existence of low rank solutions to SDP,
	Burer \& Monteiro~\cite{burer2003nonlinear}
	proposed to solve the optimization problem~\cref{eqn:burer-monteiro}
	where the variable $F \in \reals^{n \times p}$ is constrained to be a \emph{tall} matrix ($p \ll n$).
	The number $p$ is called the factorization rank.  It is clear that
	every rank-$r$ solution to the SDP~\cref{p} induces a solution
	to the factorized problem~\cref{eqn:burer-monteiro} when $p \geq r$.
	Burer \& Monteiro applied a limited-memory BFGS algorithm to solve~\cref{eqn:burer-monteiro}
	in an explicit effort to reduce storage costs.

	In subsequent work, Burer \& Monteiro~\cite{burer2005local} 
	proved that,
	under technical conditions, the local minima of
	the nonconvex formulation~\cref{eqn:burer-monteiro}
	are global minima of the SDP~\cref{p}, provided that the factorization rank
	$p$ satisfies $\binom{p+1}{2} \geq m + 1$.  As a consequence,
	algorithms based on~\cref{eqn:burer-monteiro} often set the factorization rank $p \approx \sqrt{2m}$,
	so the storage costs are $\Omega( n \sqrt{m} )$.

	Unfortunately, a recent result of Waldspurger \& Walters~\cite[Theorem~2 \& Remark~2]{waldspurger2018rank}
	demonstrates that the formulation~\cref{eqn:burer-monteiro} cannot
	lead to storage-optimal algorithms for interesting SDPs which are verified to be
	regular in \cite{LM2020regular}. In particular,
	suppose that the feasible set of~\cref{p} satisfies a mild technical
	condition and contains a matrix with rank \emph{one}.  Whenever the factorization
	rank satisfies $\binom{p+1}{2} + p \leq m$, there is a set of cost matrices
	$C$ with positive Lebesgue measure for which the factorized problem~\cref{eqn:burer-monteiro}
	has
	(1) a unique global optimizer with rank one and
	(2) at least one suboptimal local minimizer,
	while the original SDP has a unique primal and dual solution
	that satisfy strict complementarity.
	In this situation, the variable in the factorized SDP actually
	requires $\Omega(n \sqrt{m})$ storage, which is not optimal if $m = \omega(1)$.
	In view of this negative result, we omit a detailed review of the literature
	on the analysis of factorization methods.  See~\cite{waldspurger2018rank}
	for a full discussion.

	%

	\section{Basics and Notation} \label{sec:basics}

	Here we introduce some additional notation,
	and metrics for evaluating the quality of a solution and
	the conditioning of an SDP.

	\subsection{Notation}

	We will work with the Frobenius norm $\fronorm{\cdot}$,
	the $\ell_2$ operator norm $\opnorm{\cdot}$,
	and its dual, the $\ell_2$ nuclear norm $\nucnorm{\cdot}$.
	We reserve the symbols $\|\cdot\|$ and $\twonorm{\cdot}$ for the norm
	induced by the canonical inner product of the
	underlying real vector space \footnote{For symmetric matrices, we regard the trace
	inner product as the canonical one. For the Cartesian product $\sym^\dm \times \reals^\cons$,
we regard the sum of trace inner product on $\sym^\dm$ and the dot product on $\reals^\cons$ as the canonical one.}.

	For a matrix $B\in\reals^{d_1\times d_2}$,
	we arrange its singular values in decreasing order:
	$$
	\sigma_{1}(B)\geq\dots \geq \sigma_{\min(d_1,d_2)}(B).
	$$
	Define $\sigma_{\min} (B)=\sigma_{\min(d_1,d_2)}(B)$ and $\sigma_{\max}(B) = \sigma_1(B)$.
	We also write $\sigma_{\min>0}(B)$ for the smallest \emph{nonzero} singular value of $B$.
	For a linear operator $\mathcal{B} : \sym^{d_1} \to \reals^{d_2}$, we define
	$$
	\sigma_{\min}(\mathcal{B}) = \min_{\|A\|=1}\|\mathcal{B}(A)\| \quad
	\text{and}\quad  \opnorm{\mathcal{B}} = \max_{\|A\|=1}\|\mathcal{B}(A)\|.
	$$

	We use analogous notation for the eigenvalues of a symmetric matrix.
	In particular, the map $\lambda_i(\cdot):\sym^{\dm}\rightarrow \reals$
	reports the $i$th largest eigenvalue of its argument.

	\subsection{Optimal Solutions}
	\label{sec:opt-solns}


	Instate the notation and regularity assumptions from \Cref{sec:assumptions}.
	Define the slack operator $Z: \reals^\dm \to \sym^\dm$
	that maps a putative dual solution $y \in \reals^m$
	to its associated slack matrix $Z(y) := C - \mathcal{A}^\top y$. We omit the dependence on
	$y$ if it is clear from the context.

	Let the rank of primal solution being $\rsol$ and denote its range as $\vspacex$. We also fix an orthonormal matrix
	$\vrepox \in \reals^{n \times \rsol}$ whose columns span $\vspacex$.
	Introduce the subspace $\uspacex = \range(\zsol)$,
	and let $\urepox \in \reals^{n \times (n-\rsol)}$ be a fixed orthonormal basis
	for $\uspacex$.  We have the decomposition $\vspacex + \uspacex = \reals^n$.

	For a matrix $V \in \reals^{n \times r}$, define the compressed
	cost matrix and constraint map
	\begin{equation} \label{eqn:reduced-constraint}
	C_V := V^\top C V
	\quad\text{and}\quad
	\mathcal{A}_V(S) := \mathcal{A}(VSV^\top)
	\quad\text{for $S \in \sym^r$.}
	\end{equation}
	In particular, $\mathcal{A}_{\vrepox}$ is the compression of the constraint map
	onto the range of $\xsol$.


	\subsection{Conditioning of the SDP}

	Our analysis depends on conditioning properties of the pair
	of primal~\cref{p} and dual~\cref{d} SDPs.

	First, we measure the strength of the complementarity
	condition~\cref{eqn:complementary-slackness} using the
	spectral gaps of the primal solution $\xsol$ and dual slack matrix $\zsol$:
	$$
	\lambda_{\min >0}(\xsol)
	\quad\text{and}\quad
	\lambda_{\min>0}(\zsol)
	$$
	These two numbers capture how far we can perturb the solutions
	before the complementarity condition fails.

	Second, we measure the robustness of the primal solution
	to perturbations of the problem data $b$ using the quantity
	\begin{equation} \label{eqn:kappa}
	\kappa := \frac{\sigma_{\max}(\mathcal{A})}{\sigma_{\min}(\mathcal{A}_{V^\star})}.
	\end{equation}
	This term arises because we have to understand the conditioning of the system
	$\mathcal{A}_{\vrepox}(S) = b$ of linear equations in the variable $S \in \sym^{\rsol}$.

	\begin{table}[t]{\footnotesize
			\caption{Quality of a primal matrix  $X\in \sym_+^n$ and a dual vector
				$y\in\reals^m$}\label{tb:
				primaldualapproxquality}}
		\begin{center}
			\begin{tabular}{|c| c| c|}
				\hline
				&		primal matrix $X$	 &  dual vector $y$\\
				\hline
				suboptimality $(\epsilon)$ &	$\tr(CX)-\pval$ &$ \dval
				- \bvec^\top y$\\
				\hline
				infeasibility $(\delta)$	&	$\max\{\norm{\Amap X-b},(-\lambda_{\min}(X))_+\}$  & $
				(-\lambda_{\min}(Z(y)))_+$\\
				\hline
				distance to solution $(d)$	&
				$\fronorm{X-X^\star }$ &$
				\twonorm{y-\ysol}$\\
				\hline
			\end{tabular}
		\end{center}
	\end{table}

	\subsection{Quality of Solutions}\label{sec: qualityOfsolutions}

	We measure the quality of a primal matrix variable $X \in \sym_+^n$
	and a dual vector $y \in \reals_m$ in terms of their suboptimality,
	their infeasibility, and their distance to the true solutions.
	\Cref{tb: primaldualapproxquality} gives formulas for
	these quantities.

	We say that a matrix $X$ is an $(\epsilon,\delta)$-solution of~\cref{p} if its
	suboptimality $\epsilon_p(X)$ is at most $\epsilon$ and its infeasibility $\delta_p(X)$
	is at most $\delta$.

	The primal suboptimality $\epsilon_p(X)$ and infeasibility $\delta_p(X)$ are both
	controlled by the distance $d_p(X)$ to the primal solution:
	\begin{equation} \label{eqn:subopt-infeas-dist}
	\epsilon_p(X) \leq \fronorm{C}d_p(X) \quad \text{and} \quad
	\delta_p(X) \leq \max\{1, \opnorm{\Amap}\} d_p(X).
	\end{equation}
	We can also control the distance of a dual vector $y$ and its
	slack matrix $Z(y)$ from their optima using the following quadratic growth lemma.

		\begin{lemma}[Quadratic Growth]\label{lemma: qg}
		Instate the regularity assumptions from \Cref{sec:assumptions}.
		For any dual feasible $y$ with dual slack matrix
		$Z(y) := C-\mathcal{A}^\top y$ and dual suboptimality $\epsilon =\epsilon_d(y) = \dval-b^\top y$,
		we have
		\begin{equation}
			\begin{aligned} \label{eq: dqg}
				\|(Z(y),y)-(\zsol,\ysol)\| \leq\frac{1}{ \sigma_{\min}(\mathcal{D})}
				\biggr[\frac{\epsilon}{\lambda_{\min>0}(\xsol)}+
				\sqrt{\frac{2\epsilon}{\lambda_{\min>0}(\xsol)}\opnorm{Z(y)}}\biggr],
			\end{aligned}
		\end{equation}
		where the linear operator  $\mathcal{D}: \sym^n \times \reals^{m}
		\rightarrow \sym^{n} \times \sym^n$ is defined by
		$$
		\mathcal{D} (Z,y) := (Z- (\urepox \urepox^\top) Z (\urepox \urepox^\top), Z+\mathcal{A}^\top y).
		$$
		The orthonormal matrix $\urepox$ is defined in~\Cref{sec:opt-solns}. The quantity $\sigma_{\min}(\mathcal{D})$ is defined as
		$
		\sigma_{\min} (\mathcal{D})\defn  \min_{ \norm{(Z,y)}=1}
		\norm{(Z- (\urepox \urepox^\top) Z (\urepox \urepox^\top),Z+\mathcal{A}^\top y)}.
		$
	\end{lemma}

	The proof of~\cref{lemma: qg} can be found in~\Cref{sec: lemmaOfIntroduction}.
	The name \emph{quadratic growth} arises from a limit of inequality \cref{eq: dqg}:
	when $\epsilon$ is small,
	the second term in the bracket dominates the first term,
	so $\twonorm{y-\ysol}^2 = \bigO{\epsilon}$~\cite{drusvyatskiy2018error}.

	\section{Reduced SDPs and Approximate Complementarity}\label{sec: oracleinequalities}


	In this section, we describe two reduced SDP formulations,
	and we explain when their solutions are nearly optimal
	for the original SDP~\cref{p}.  We can interpret these results as
	constructive proofs of the approximate complementarity principle.

	\subsection{Reduced SDPs}
	\label{sec:reduced-sdps}

	Suppose that we have obtained a dual approximate solution $y$
	and its associated dual slack matrix $Z(y) := C - \mathcal{A}^\top y$.
	Let $r$ be a rank parameter, which we will discuss later.
	Construct an orthonormal matrix $V \in \reals^{n \times r}$
	whose range is an $r$-dimensional invariant subspace associated with the
	$r$ smallest eigenvalues of the dual slack matrix $Z(y)$.
	Our goal is to compute a matrix $X$ with range $V$
	that approximately solves the primal SDP~\cref{p}.

	Our first approach minimizes infeasibility over all psd matrices with range $V$:
	\beq\tag{\minfeas}
	\ba{ll}
	\mbox{minimize} &\frac{1}{2}\norm{\Av{V} (S ) -  b}^2\\
	\mbox{subject to} & S\succeq 0,
	\ea
	\eeq
	with variable $S \in \sym^r$.  Given a solution $\hat{S}$, 
	we can form an approximate solution $\infeasx=V\hat{S}V^\top$ for the primal SDP~\cref{p}.
	This is the same method from~\Cref{sec:approx-comp-intro}.

	Our second approach minimizes the objective value over all psd matrices with range $V$,
	subject to a specified limit $\delta$ on infeasibility:
	\beq \label{minobj} \tag{\minobj}
	\ba{ll}
	\mbox{minimize} & \tr( C_V S)\\
	\mbox{subject to} & \|\Av{V}(S) -  b\|\leq \delta
	\quad\text{and}\quad S\succeq 0, \\
	\ea
	\eeq
	with variable $S \in \sym^r$.
	Given a solution $\tilde{S}$, we can form an approximate solution
	$\objx = V\tilde{S}V^\top$ for the primal SDP~\cref{p}.

	As we will see, both approaches lead to satisfactory solutions to the original SDP~\cref{p}
	under appropriate assumptions.
	\Cref{thm: minfeas} addresses the performance of~\cref{minfeas},
	while \cref{thm: minobj} addresses the performance of~\cref{minobj}.
	\Cref{tb: comparison} summarizes the hypotheses
	we impose to study each of the two problems, as well as the outcomes of the analysis.

	The bounds in this section depend on the problem data
	and rely on assumptions that are not easy to check.
    We discuss how to check the quality of $\infeasx$ and $\objx$.
    in \Cref{sec: algorithmguarantee}.

	\jatnote{Warning: I did not check *any* of the math this time.  The proofs could
	probably be smoothed down, and the typesetting needs more work.  I would also
	add a bit more discussion about what the theorems say, in the locations that I indicated.}

	\subsection{Analysis of \cref{minfeas}}\label{sec: minfeas}

	First, we establish a result that connects the solution
	of \cref{minfeas} with the solution of the original problem~\cref{p}.



	\begin{theorem}[Analysis of \cref{minfeas}]\label{thm: minfeas}
		Instate the regularity assumptions in \Cref{sec:assumptions}.
		Moreover, assume
		the solution rank $\rsol$ is known. Set $r = \rsol$.
		Let $y \in \reals^m$ be feasible for the dual SDP~\cref{d}
		with suboptimality $\epsilon = \epsilon_d(y) = \dval- b^\top y <  c_1$,
		where the constant $c_1 > 0$ depends only on $\Amap,\bvec$ and $\objc$.
		Then the threshold $T := \lambda_{n-r}(Z(y))$ obeys
		$$
		T := \lambda_{n-r}(Z(y))\geq \frac{1}{2}\lambda_{n-r}(\zsol)>0,
		$$
		and we have the bound
		\begin{align}\label{distanceboundint0}
			\fronorm{\infeasx  - \xsol} \leq (1+ 2\kappa)\bigg( \frac{\epsilon}{T} +
			\sqrt{2\frac{\epsilon}{T} \opnorm{\xsol}}\bigg).
		\end{align}
	\end{theorem}

	\jatnote{*This* is where we discuss the result in as much detail as you want.}

	This bound shows that
	$\fronorm{\infeasx  - \xsol }^2 = \bigO{\epsilon}$
	when the dual vector $y$ is $\epsilon$ suboptimal.
	Notice this result requires knowledge of the solution rank $\rsol$.
	The proof of~\cref{thm: minfeas} occupies the rest of this section.



	\subsubsection{Primal Optimizers and the Reduced Search Space}

	The first step in the argument is to prove that $\xsol$ is near the search space
	$\{VSV^\top :  S\in \sym^{r}_+\}$ of the reduced problems. 
	\begin{lemma}\label{lemma: minfeas0}
	Instate the regularity assumptions in \Cref{sec:assumptions}.
    Further suppose $y \in \reals^m$ is feasible and $\epsilon$-suboptimal
		for the dual SDP \cref{d}, and construct the orthonormal matrix $V$
		as in \Cref{sec:reduced-sdps}.
		Assume that the threshold $T := \lambda_{n-r}(C-\mathcal{A}^\top y) > 0$. Define 
		$P_{V}(X)=VV^\top XVV^\top$, and $P_{V^\perp}(X)=X-P_V(X)$ for any $X\in \sym^\dm$.
		Then for	any solution $\xsol$ of the primal SDP \cref{p},
		$$
		\fronorm{P_{V^\perp}(\xsol)} \leq  \frac{\epsilon}{T} +
		\sqrt{2\frac{\epsilon}{T} \opnorm{\xsol}},\;\text{and}\;\nucnorm{P_{V^\perp}(\xsol)
	} \leq  \frac{\epsilon}{T} + 2\sqrt{r\frac{\epsilon}{T} \opnorm{\xsol}}.
		$$
	\end{lemma}
To prove the lemma, we will utilize the following result (proved in \Cref{sec: lemmasForSectionOfOracleinequalities}) which bounds the distance
to subspaces via the inner product. This result might be of independent interest.
\begin{lemma}\label{lem: MainProjectionDistanceInnerProduct}
	Suppose $X,Z\in \sym^{\dm}$ are both positive semidefinite.
	Let $V\in \reals^{\dm\times r}$ be the matrices formed by the eigenvectors
	with the smallest $r$ eigenvalues of $Z$ . Let $\epsilon =\tr(XZ)$ and
	$P_V(X)=VV^\top XVV^\top$.
	If $T = \lambda_{n-r}(Z)>0$, then
	\[
	\fronorm{X-P_V(X) } \leq  \frac{\epsilon}{T} +
	\sqrt{2\frac{\epsilon}{T} \opnorm{X}},\;\text{and}\; \nucnorm{X-P_V(X)
	} \leq  \frac{\epsilon}{T} + 2\sqrt{r\frac{\epsilon}{T} \opnorm{X}}.
	\]
\end{lemma}
Now we are ready to prove \Cref{lemma: minfeas0}.
\begin{proof}[Proof of \Cref{lemma: minfeas0}]
	We shall utilize \Cref{lem: MainProjectionDistanceInnerProduct}. Simply set $Z$ in \Cref{lem: MainProjectionDistanceInnerProduct}
	to be $C-\Amap^\top y$ from the approximate dual solution $y$, and $X$ to be the primal solution $\xsol$. Using
	strong duality in the following step $(a)$ and feasibility of $\xsol$ in the following step $(b)$, we have
	\[
	\epsilon = b^\top \ysol -b^\top y \overset{(a)}{=} \tr(C\xsol)-b^\top y \overset{(b)}{=} \tr(C\xsol) - (\Amap \xsol)^\top y=
	\tr(\xsol Z).
	\]
	Hence, we can apply Lemma \ref{lem: MainProjectionDistanceInnerProduct} to obtain the bounds in \Cref{lemma: minfeas0}.
\end{proof}
	\subsubsection{Relationship between the Solutions of \cref{minfeas} and \cref{p}}

	\Cref{lemma: minfeas0} shows that any solution $\xsol$ of \cref{p}
	is close to its compression $VV^\top \xsol VV^\top$ onto the range of $V$.
	Next, we show that $\infeasx$ is also close to $VV^\top \xsol VV^\top$.
	We can invoke strong convexity of the objective of~\cref{minfeas}
	to achieve this goal.



	\begin{lemma}\label{lemma: minfeas1}
	Instate the assumptions and notation from \cref{lemma:
		minfeas0}.
	Assume $\sigma_{\min} (\mathcal{A}_V)>0$. and
		 that the threshold $T = \lambda_{n-r}(Z(y)) > 0$.
		Then
		\begin{align}
			\fronorm{\infeasx - \xsol } \leq \left(1+
			\frac{\sigma_{\max}(\mathcal{A})}{\sigma_{\min}(\mathcal{A}_V)}\right)
			\left(
			\frac{\epsilon}{T} + \sqrt{2\frac{\epsilon}{T}
				\opnorm{\xsol}}\right),\label{fineq}
		\end{align}
		where $\xsol$ is any solution of the primal SDP \cref{p}.
	\end{lemma}
	\begin{proof}
		Since we assume that $\sigma_{\min} (\Av{V})>0$,
		we know the objective of \cref{minfeas},
		$f(S)= \frac{1}{2}\twonorm{\Av{V}(S)-b}^2$, is
		$\sigma_{\min}^2(\mathcal{A}_V)$-strongly convex,
		and so the solution $\ssol$ is unique.
		We then have for any $S\in \sym^r$
		\begin{equation}
			\begin{aligned}
				f(S) - f(S^\star) & \overset{(a)}{\geq} \tr(\nabla f (S^\star)^\top ( S -
				S^\star)) + \frac{\sigma^2_{\min}(\Av{V})}{2}\fronorm{S-S^\star}^2\\
				&\overset{(b)}{\geq}
				\frac{\sigma_{\min}^2(\mathcal{A}_V)}{2}\fronorm{S-S^\star}^2,\label{eq:
					lemmaminfeas1e1}
			\end{aligned}
		\end{equation}
		where step $(a)$ uses strong convexity
		and step $(b)$ is due to the optimality of $\ssol$.

		Since $\mathcal{A}\xsol = b $, we can bound the objective of \cref{minfeas}
		by $\fronorm{ P_V(X) - \xsol}$:
		\begin{equation}
			\begin{aligned}
				\twonorm{ \Av{V}(V^\top XV) - b} &=  \twonorm{\Amap(P_V(X) - \xsol)}  
				&\leq
				\sigma_{\max}(\mathcal{A}) \fronorm{P_V(X) - \xsol}.\label{eq:
					lemmaminfeas1e2}
			\end{aligned}
		\end{equation}
		Combining pieces, we know that $\ssol$ satisfies
		\begin{align*}
			\fronorm{S^\star - V^\top \xsol V}^2
			&\ineqov{(a)}\frac{2}{\sigma_{\min}^2(\mathcal{A}_V)}( f(V^\top \xsol V)  -
			f(S^\star))
			\ineqov{(b)}\frac{\sigma^2_{\max}(\mathcal{A})}{\sigma_{\min}^2(\mathcal{A}_V)}
			\fronorm{\xsol - P_V(\xsol)}^2 \\
			&\ineqov{(c)}\frac{\sigma^2_{\max}(\mathcal{A})}{\sigma_{\min}^2(\mathcal{A}_V)}
			\bigg( \frac{\epsilon}{T} + \sqrt{2\frac{\epsilon}{T} \opnorm{\xsol}}\bigg)
			^2,
		\end{align*}
		where step $(a)$ uses \cref{eq: lemmaminfeas1e1},
		step $(b)$ uses \cref{eq: lemmaminfeas1e2} for $X=\xsol$ and $f(S^\star)\geq 0$,
		and step $(c)$ uses \cref{lemma: minfeas0}.
		Lifting to the larger space $\reals^{n \times n}$, we see
		\begin{align*}
			\fronorm{V S^\star V^\top - \xsol} & \leq \fronorm{VS^\star V^\top -P_V(\xsol)}+
			\fronorm{\xsol -P_V(X)} \\
			& \overset{(a)}{=} \fronorm{S^\star -V^\top \xsol V} + \fronorm{\xsol -P_V(\xsol)
			}\\
			& \ineqov{(b)} \left(1 +  \frac{
				\sigma_{\max}(\mathcal{A})}{\sigma_{\min}(\mathcal{A}_V)}\right) \bigg(
			\frac{\epsilon}{T} + \sqrt{2\frac{\epsilon}{T} \opnorm{\xsol}}\bigg) .
		\end{align*}
		Here we use the unitary invariance of $\fronorm{\cdot}$ in $(a)$.
		The inequality $(b)$ is due to our bound above for $\ssol$ and
		\cref{lemma: minfeas0}.
	\end{proof}

	\subsubsection{Lower Bounds for the Threshold and Minimum Singular Value}

	Finally, we must confirm that the extra hypotheses of \cref{lemma: minfeas1} hold,
	\ie $T > 0$ and $\sigma_{\min} (\mathcal{A}_V)>0$.

	We explain the intuition here.
	Strict complementarity forces $\lambda_{n-r}(\zsol) > 0$.
	If $Z$ is close to $\zsol$, then we expect that $T > 0$ by continuity.
	When $\xsol$ is unique, \Cref{thm: uniqueness} implies that $\nullspace(\mathcal{A}_\vrepox)=\{0\}$.
	As a consequence, $\sigma_{\min}(\mathcal{A}_{\vrepox})>0$.
	If $V$ is close to $\vrepox$, then we expect that $\sigma_{\min}(\mathcal{A}_V) > 0$
	as well.
	We have the following rigorous statement.


	\begin{lemma}\label{lemma: minfeas2}
		Instate the hypotheses of~\cref{thm: minfeas}.  Then
		$$
		\begin{aligned}
		T =\lambda_{n-r}(Z(y)) &\geq \frac{1}{2}\lambda_{n-r}(\zsol); \\
		\sigma_{\min} (\mathcal{A}_V) &\geq \frac{1}{2}
		\sigma_{\min}(\mathcal{A}_{V^\star})>0.
		\end{aligned}
		$$
	\end{lemma}
	\begin{proof} We first prove the lower bound on the threshold $T$.
		Using
		$\|(Z,y)-(\zsol,\ysol)\|\geq \opnorm{Z-\zsol}\geq \opnorm{Z}-\opnorm{\zsol}$
		and quadratic growth (\cref{lemma: qg}), we have
		$$\opnorm{Z}-\opnorm{\zsol} \leq\frac{1}{ \sigma_{\min}(\mathcal{D})}
		\biggr(\frac{\epsilon}{\lambda_{\min>0}(\xsol)}+
		\sqrt{\frac{2\epsilon}{\lambda_{\min>0}(\xsol)}\opnorm{Z}}\biggr).   $$
		Thus for sufficiently small $\epsilon$,
		we have $\opnorm{Z} \leq 2\opnorm{\zsol}$.
		Substituting this bound into previous inequality gives,
		\begin{equation}\label{eq: lemma: minfeas2eq1}
			\begin{aligned}
				\|(Z,y)-(\zsol,\ysol)\|\leq\frac{1}{ \sigma_{\min}(\mathcal{D})}
				\biggr(\frac{\epsilon}{\lambda_{\min>0}(\xsol)}+
				\sqrt{\frac{4\epsilon}{\lambda_{\min>0}(\xsol)}\opnorm{\zsol}}\biggr).
			\end{aligned}
		\end{equation}
		Weyl's inequality tells us that
		$\lambda_{n-r} (\zsol)-T \leq \opnorm{Z-\zsol}$.
		Using \cref{eq: lemma: minfeas2eq1},
		we see that for all sufficiently small $\epsilon$,
		$T:=\lambda_{n-r}(C-\mathcal{A}^\top y)\geq \frac{1}{2}\lambda_{n-r}(\zsol).$

		Next we prove the lower bound on $\Av{V}$.  We have
		$\sigma_{\min}(\mathcal{A}_{V^\star})>0$  by \cref{thm: uniqueness}.
		It will be convenient to align the columns of $V$ with those of $V^\star$
		for our analysis.
		Consider the solution $O_\star$ to the orthogonal Procrustes problem
		$O_\star =	\argmin_{OO^\top =I, O\in \reals^{r\times r}}\fronorm{VO-V^\star}$.
		Since $\sigma_{\min}(\mathcal{A}_V) = \sigma_{\min}(\mathcal{A}_{VO_\star})$ for
		orthonormal
		$O_\star$, without loss of generality,
		we suppose we have already performed the alginment and $V$ is $VO_\star$ in the following.

		Let $S_1= \argmin_{\fronorm{S}=1}\twonorm{\Av{V}(S)}$. Then we have
		\beq
		\begin{aligned}\label{l2e1}
			\sigma_{\min}(\Av{\vrepox}) - \sigma_{\min} (\Av{V} )  &\leq
			\twonorm{\Av{\vrepox}(S_1)}-\twonorm{\Av{V}(S_1)} \\
			&\leq  \twonorm{\mathcal{A}(V^\star S_1 (V^\star)^\top )-\mathcal{A}(VS_1V^\top )
			}\\
			&\leq \opnorm{\Amap}\fronorm{ V^\star S_1 (V^\star)^\top -(VS_1V^\top )}.
		\end{aligned}
		\eeq
 	  Defining $E =V-  V^\star $, we bound the term $ \fronorm{V^\star S_1
			(V^\star)^\top -(VS_1V^\top )}$ as
		\beq
		\ba{ll}\label{l2e2}
		\fronorm{V^\star S_1 (V^\star)^\top -(VS_1V^\top )} & =
		\fronorm{ES_1(\vrepox)^\top+\vrepox S_1E^\top + ES_1E^\top}\\
		& \ineqov{(a)} 2\fronorm{E}\fronorm{\vrepox S_1} + \fronorm{E}^2
		\fronorm{S_1}\\
		& \overset{(b)}{=} 2\fronorm{E} + \fronorm{E}^2,
		\ea
		\eeq
		where $(a)$ uses the triangle inequality and the submultiplicativity of
		the Frobenius norm.
		We use the orthogonality of the columns of $V$ and of $V^\star$
		and the fact that $\fronorm{S_1}=1$ in step $(b)$.

		A variant of the Davis--Kahan inequality~\cite[Theorem 2]{yu2014useful} asserts that
		$ \fronorm{ E} \leq 4\fronorm{Z-\zsol}/\lambda_{\min>0}(\zsol)$.
		Combining this fact with inequality \cref{eq: lemma: minfeas2eq1}, we see
		$\opnorm{E}\rightarrow 0$ as $\epsilon \rightarrow 0$.
		Now using \cref{l2e2} and \cref{l2e1}, we see that
		for all sufficiently small $\epsilon$,
		$\sigma_{\min} (\mathcal{A}_V) \geq \frac{1}{2}\sigma_{\min}(\mathcal{A}_{V^\star})>0.$
			\end{proof}


	\subsubsection{Proof of \cref{thm: minfeas}}


	Instate the hypotheses of~\cref{thm: minfeas}.  Now, \cref{lemma: minfeas2} implies
	that $\sigma_{\min}(\mathcal{A}_V) > 0$ and that $T > 0$.  Therefore,
	we can invoke~\cref{lemma: minfeas1} to obtain the stated bound on
	$\fronorm{\infeasx - \xsol }$.

	\subsection{Analysis of \cref{minobj}}\label{sec: minobj}

	Next, we establish a result that connects the solution to~\cref{minobj}
	with the solution to the original problem~\cref{p}.

	\begin{theorem}[Analysis of \cref{minobj}] \label{thm: minobj}
		Instate the regularity assumptions in \Cref{sec:assumptions}.
		Moreover, assume $r \geq \rsol$.
		Let $y \in \reals^m$ be feasible for the dual SDP~\cref{d}
		with suboptimality $\epsilon = \epsilon_d(y) = \dval- b^\top y <  c_2$,
		where the constant $c_2 > 0$ depends only on $\Amap,\bvec$ and $\objc$.
		Then the threshold $T := \lambda_{n-r}(Z(y))$ obeys
		$$
		T := \lambda_{n-r}(Z(y)) \geq \frac{1}{2}\lambda_{n-r}(\zsol) > 0.
		$$
		Introduce the quantities
		\[
		\begin{aligned}
		\delta_0 &:=
		\sigma_{\max}(\mathcal{A})\left(\frac{\epsilon}{T} +
		\sqrt{2\frac{\epsilon}{T}
		\opnorm{\xsol}}\right); \\
		\epsilon_0 &:=\min\left\{
		\fronorm{C}
		\left(\frac{\epsilon}{T} + \sqrt{2\frac{\epsilon}{T} \opnorm{\xsol}}\right),
		\opnorm{C} \left(\frac{\epsilon}{T} + \sqrt{2\frac{r \epsilon}{T}
			\opnorm{\xsol}}\right)
			\right\}.
		\end{aligned}
		\]
		If we solve~\cref{minobj} with the infeasibility parameter $\delta = \delta_0$,
		then the resulting matrix $\objx$ is an $(\epsilon_0, \delta_0)$
		solution to~\cref{p}.

		If in addition $C = I$, then $\objx$ is
		superoptimal with $0\geq \epsilon_0 \geq -\frac{\epsilon}{T}.$
	\end{theorem}

	\jatnote{*This* is where we give a detailed discussed of the result.}
		The analysis \cref{thm: minfeas} of \cref{minfeas}
		requires knowledge of the solution rank $\rsol$, and the bounds
		depend on the conditioning $\kappa$.
	In contrast, \cref{thm: minobj} does not require knowledge of $\rsol$, and the bounds do not depend on $\kappa$.
		\cref{tb: comparison} compares our findings for the two
		optimization problems \cref{thm: minfeas} and \cref{thm: minobj}.

	\begin{remark}\label{Remark2: mainthm2}
		The quality of the primal reconstruction depends
		on the ratio between the threshold $T$ and the suboptimality $\epsilon$.
		The quality improves as the suboptimality $\epsilon$ decreases,
		so the primal reconstruction approaches optimality
		as the dual estimate $y$ approaches optimality.
		The threshold $T$ is increasing in the rank estimate $r$, and so
		the primal reconstruction improves as $r$	increases.
		Since $r$ controls the storage required for the primal reconstruction,
		we see that the quality of the primal reconstruction improves
		as our storage budget increases.
	\end{remark}
	\begin{remark}\label{Remark: mainthm2}
		Using the concluding remarks of \cite{sturm2000error},
		the above bound on suboptimality and infeasibility shows that
		the distance between $\objx$ and $\xsol$ is at most
		$\bigO{\epsilon^{1/4}}$. Here, the
		$\bigO\cdot$ notation omits constants depending on $\mathcal{A}$, $b$, and $C$.
	\end{remark}
	\begin{table}{\footnotesize
			\caption{Comparison of \cref{minfeas} and \cref{minobj} given a feasible
				$\epsilon$-suboptimal dual vector $y$.}\label{tb: comparison}}
		\begin{center}
			\begin{tabular}{|c| c|c|}
				\hline
				Assumption and Quality	&
				 \cref{minfeas} & \cref{minobj} \\
				 \hline
				\hline
				Require $r=\rsol$ ? &Yes & No\\
				\hline
				Suboptimality &  $\mathcal{O}(\kappa\sqrt{\epsilon})$ &
				$\mathcal{O}(\sqrt{\epsilon})$ \\
				\hline
				Infeasibility & $\mathcal{O}(\kappa\sqrt{\epsilon})$ &
				$\mathcal{O}(\sqrt{\epsilon})$\\
				\hline
				Distance to the solution &  $\mathcal{O}(\kappa\sqrt{\epsilon})$ & \cref{Remark:
					mainthm2}\\
				\hline
			\end{tabular}
		\end{center}
	\end{table}

	The proof of \cref{thm: minobj} occupies the rest of this subsection.

\subsubsection{Bound on the Threshold via Quadratic Growth}

	We first bound $T$ when the suboptimality of $y$ is bounded.
	This bound is a simple consequence of quadratic growth (\cref{lemma: qg}).

	\begin{restatable}{lemma}{eigenvalue}\label{eigenvalueb}
	Instate the hypotheses of \cref{thm: minobj}.
		Then
		$$
		T := \lambda_{n-r}(Z(y)) \geq \frac{1}{2}\lambda_{n-r}(\zsol)>0.
		$$
	\end{restatable}

\begin{proof}
	The proof follows exactly the same line (without even changing of the notation)
	as the proof of \cref{lemma: minfeas2} in
	assuring $\lambda_{n-\rsol}(Z(y))>0$, by noting
	$\lambda_{n-r}(\zsol)>0$ whenever $r \geq \dm -\rank(\zsol)$.
	\end{proof}

\subsubsection{Proof of \cref{thm: minobj}}

	\Cref{lemma: minfeas0}
	shows that any primal solution $\xsol$,
	is close to $VV^\top \xsol VV^\top=:P_V(\xsol)$.
	We must ensure that $P_V(\xsol)$ is feasible for \cref{minobj}.
	This is achieved by setting the infeasibility parameter in~\cref{minobj}
	as
	$$
	\delta := \sigma_{\max}(\mathcal{A})\biggr(\frac{\epsilon}{T} +
	\sqrt{2\frac{\epsilon\opnorm{\xsol}}{T}}\biggr)
	$$
	This choice also guarantees all solutions to \cref{minobj}
	are $\delta$-feasible.

		The solution to \cref{minobj} is $\delta_0$-feasible by construction.
		It remains to show the solution is $\epsilon_0$-suboptimal.
		We can bound the suboptimality of the feasible point $P_V(\xsol)$
		to produce a bound on the suboptimality of the solution to \cref{minobj}.
		We use H\"older's inequality to translate the bound on the
		distance between $P_V(\xsol)$ and $\xsol$, from \cref{lemma: minfeas0},
		into a bound on the suboptimality:
		\begin{multline*}
		\tr(C (P_V(\xsol)- \xsol )) \\
		\leq
		\epsilon_0 := \min\biggr\{
		\fronorm{C} \biggr(\frac{\epsilon}{T} + \sqrt{2\frac{\epsilon}{T}
			\opnorm{\xsol}}\biggr)
		, \opnorm{C} \biggr(\frac{\epsilon}{T} + \sqrt{2\frac{r \epsilon}{T}
			\opnorm{\xsol}}\biggr) \biggr\}.
		\end{multline*}
		This argument shows that
	$P_V(\xsol)$ is feasible, and hence the solution to \cref{minobj}, is at most $\epsilon_0$ suboptimal.

To prove the improvement for the case $C=I$, we first
complete $V$ to form a basis $W =[U\,V]$ for $\reals^n$,
where $U=[v_{r+1},\dots ,v_n]\in \reals^{n\times (n-r)}$ and where
$v_{i}$ is the eigenvector of $Z$ associated with the $i$-th
smallest eigenvalue. Define $X_1 = U^\top \xsol U$ and $X_2 = V^\top \xsol V$.
We first note that
\begin{equation*}
\begin{aligned}
\tr(\xsol)=\tr(W^\top \xsol W) = \tr(X_1) + \tr(X_2), \quad\text{and}\quad \tr(X_2) = \tr(VV^\top \xsol VV^\top).
\end{aligned}
\end{equation*}
We can bound $\tr(X_1)$ using the following inequality:
\begin{equation*}
\begin{aligned}
\epsilon \overset{(a)}{=}\tr(Z\xsol)
= \sum_{i=1}^n\lambda_{n-i+1}(Z) v_i^\top \xsol v_i\overset{(b)}{\geq} T \sum_{i=r+1}^n v_i^\top \xsol  v_i
= \tr(U\xsol U^\top)
=\tr(X_1).
\end{aligned}
\end{equation*}
Here step $(a)$ is due to strong duality and we uses $v_i^\top \xsol v_i\geq 0$ in step $(b)$ as $\xsol \succeq 0$.
Combing pieces and $\tr(X_1)\geq 0$ as $X_1\succeq 0$, we find that
$$\tr(\xsol)\geq \tr(VV^\top \xsol VV^\top) \geq \tr(\xsol)-\dfrac{\epsilon}{T}.$$
This completes the argument.

	\section{Computational Aspects of Primal Recovery} \label{sec: algorithmguarantee}

	The previous section introduced two methods,
	\cref{minfeas} and	\cref{minobj},
 	to recover an approximate primal
	from an approximate dual solution $y$.
	It contains theoretical bounds on
	suboptimality, infeasibility, and distance
	to the solution set of the primal SDP~\cref{p}.
	We summarize this approach as
	\Cref{primalrecovery}.

	\begin{algorithm}[t]
		\caption{\label{primalrecovery}Primal recovery via \cref{minfeas} or
			\cref{minobj}}
		\begin{algorithmic}[1]
			\REQUIRE{Problem data $\mathcal{A}$, $C$ and $b$; dual vector $y$
				and positive integer $r$}
			\STATE Compute an orthonormal matrix $V \in \reals^{n \times r}$
			whose range is an invariant subspace of $C - \mathcal{A}^\top y$
			associated with the $r$ smallest eigenvalues.
			\STATE Option 1: Solve \cref{minfeas} to obtain a matrix $\hat{S}_{1} \in \sym^r_+$.
			\STATE Option 2: Solve \cref{minobj} by seting  $\delta= \gamma \twonorm{\Av{V}(\hat{S}_1)-b}$ with some $\gamma\geq 1$,
			where $\hat{S}_1$ is obtained from solving \cref{minfeas}. Obtain $\hat{S}_2$.
			\RETURN $(V,S_1)$ for option 1, and $(V,S_2)$ for option 2.
		\end{algorithmic}
	\end{algorithm}

	In this section, we turn this approach into a
	practical optimal storage algorithm,
	by answering the following questions:
	\begin{enumerate}
		\item How should we solve \cref{minfeas} and \cref{minobj}?
		\item How should we choose $\delta$ in \cref{minobj}?
		\item How should we choose the rank parameter $r$? 
		\item How can we estimate the suboptimality, infeasibility, and
		(possibly) the distance to the solution to use as stopping conditions?
	\end{enumerate}
	In particular, our choices for algorithmic parameters
	should not depend on any quantities that are unknown or difficult to compute.
	We address each question in turn.

	For this discussion, let us quantify the cost of the three data
	access oracles~\cref{eq: efficientOperation}.  We use the mnemonic
	notation $L_C$, $L_{\mathcal{A}}$, and $L_{\mathcal{A}^\top}$ for
	the respective running time (denominated in flops)
	of the three operations.


	\subsection{Solving \minfeas and \minobj}
	Suppose that we have a dual estimate $y \in \reals^m$,
	and that we have chosen $r = \bigO{\rsol}$ and $\delta$.
	Each recovery problem, \cref{minfeas} and \cref{minobj},
	is an SDP with an $r \times r$
	decision variable and $m$ linear constraints.
	We now discuss how to solve them with optimal storage $\bigO{m + nr}$.
	First, we present four operators that form the computational core of all the storage optimal algorithms we consider here.
	We list their input and output dimension, storage requirement (sum of input output dimensions), time complexity in evaluating 
	these operators in \cref{table: requirement}.
	\begin{table}
		\centering
		 \caption{Required operators for solving \minfeas and \minobj. The operators 
			$P_{\tiny \mathbf{B}_\delta}$ and  $P_{\tiny \sym_+^r}$ are projections of the $\ell_2$ norm ball $\mathbf{B}_\delta \defn \{y\in \reals^m \mid \twonorm{y}\leq \delta \}$
			and of the PSD matrices of side dimension $r$, $\sym_+^r$, respectively.
			\label{table: requirement}}
	 \begin{tabular}{|c|c|c|c|c|}
	 	\hline 
	Operator & Input & Output & Storage req. & Time Compl. \\  
	\hline 
	 $\Av{V}$	& $S\in \sym^r$ & $\Av{V}(S)\in \reals^\cons$ & $r^2+m$ &  $\bigO{r^2 L_{\mathcal{A}}}$  \\
	  $\Av{V}^\top $	&   $y\in \reals^{\cons}$ & $ V^\top(\mathcal{A}^\top(y))V\in \sym^r$ & $m+r^2$ & $ \bigO{rL_{\mathcal{A}^\top} +nr^2}$ \\
	 	
	  $P_{\tiny \mathbf{B}_\delta}$ & $y\in \reals^{\cons}$ & $P_{\tiny \mathbf{B}_\delta}(y)\in \reals^{\cons}$ & $2m$ & $\bigO{m}$ \\ 
	  $	P_{\tiny \sym_+^r}$ & $S\in \sym^r$ & $P_{\tiny \sym_+^r}(S)\in \sym^r$ & $2r^2$ & $\bigO{r^3}$\\
	 	\hline 
	 \end{tabular}
 \end{table}   
	Any algorithm that uses a constant number of calls to these operators at each iteration
	(and at most $\bigO{m + nr}$ additional storage)
	achieves optimal storage $\bigO{m + nr}$.
	To be concrete, we describe algorithms to solve \cref{minfeas} and \cref{minobj} that
	achieve optimal storage. Many other algorithmic choices are possible.
	\begin{itemize}
		\item For \cref{minfeas}, we can use the accelerated projected gradient method \cite{nesterov2013introductory}.
		This method uses the operators $\Av{V},\Av{V}^\top$, and $P_{\tiny \sym_+^r}$.
		Each iteration requires one call each to $\Av{V}$, $\Av{V}^\top$, and $P_{\tiny \sym_+^r}$,
		and a constant number of additions in $\reals^m$ and $\sym_+^r$.
		Hence the per iteration flop count is $\bigO{r^2 L_{\mathcal{A}} + rL_{\mathcal{A}^\top}
			+ m + r^2n}$.
		As for storage,
		the accelerated projected gradient method requires $\mathcal O(m + r^2)$
		working memory to store the residual
		$\Av{V}(S)-b$,
		the computed gradient,
		and iterates of size $r^2$.
		Hence this method is storage optimal.

		\item For \cref{minobj}, we can use the Chambolle-Pock method \cite{chambolle2011first}.
		We present a detailed description in \Cref{sec: cpminobj}.
		This method requires access to the operators $\Av{V},\Av{V}^\top$, $P_{\tiny \mathbf{B}_\delta}$ and $P_{\tiny \sym_+^r}$.
		It also stores the matrix $C_V = V^T C V \in \reals^{r \times r}$ explicitly.
		We can compute $C_V$ in $r^2L_{C}$ time and store it using $r^2$ storage.
		Each iteration requires one call each to $\Av{V}$, $\Av{V}^\top$,  $P_{\tiny \mathbf{B}_\delta},$ and $ P_{\tiny \sym_+^r}$,
		and a constant number of additions in $\reals^m$ and $\sym_+^r$.
		Hence the per iteration flop count is $\bigO{r^2 L_{\mathcal{A}} + rL_{\mathcal{A}^\top} + m + r^2n}$.
		As for storage, the Chambolle-Pock method requires $\mathcal O(m + r^2)$ working memory
		to store the residual $\Av{V}(S)-b$, one dual iterate of size $m$,
		two primal iterates of size $r^2$,
		and a few intermediate quantities of size $r^2$ or $m$.
		Hence the method is again storage-optimal.
	\end{itemize}

\subsection{Choosing the Rank Parameter $r$}
\cref{thm: minfeas} shows that \cref{minfeas} recovers the solution
when the rank estimate $r$ is accurate. Alas, as $r$ increases,
\cref{minfeas} can have multiple solutions.
Hence it is important to use information about the objective function as well
(\eg using \cref{minobj}) to recover the solution to \cref{p} --- in theory.
In practice, we find that \cref{minfeas} recovers the primal solution well
so long as $r$ satisfies the Barvinok-Pataki bound $\frac{r(r+1)}{2}\leq \cons$.

\cref{thm: minobj} shows that \cref{minobj} is more robust, and provides useful results
so long as the rank estimate $r$ exceeds the true rank $\rsol$.
Indeed, the quality of the solution improves as $r$ increases.
A user seeking the best possible solution to \cref{minobj} should choose
the largest rank estimate $r$ for which the SDP \cref{minobj}
can still be solved, given computational and storage limitations.

It is tempting to consider the spectrum
of the dual slack matrix $C-\mathcal{A}^\top y$,
and in particular, its smallest eigenvalues, to guess the true rank of the solution.
We do not know of any reliable rules that use this idea.

	\subsection{Choosing the Infeasibility Parameter $\delta$}
	To solve \cref{minobj}, we must choose a bound $\delta$ on the acceptable infeasibility.
	(Recall that \cref{minobj} is generally not feasible when $\delta = 0$.)
	This bound can be chosen using the result of \cref{minfeas}. 
	Concretely, solve \cref{minfeas} to obtain a solution $\infeasx$.
	Then set $\delta = \gamma\twonorm{\Av{V}(\infeasx)-b}$ for some $\gamma \geq 1$.
	This choice guarantees that \cref{minobj} is feasible.
	In our numerics, we find $\gamma =1.1$ works well. \footnote{
	It is possible to directly set the value $\delta$ without solving \cref{minfeas} using
	the bounds from \Cref{thm: computablebound} in \Cref{sec: computablebound}
	when additional information or computation budget is available. However, evaluating the bounds (which can be potentially loose)
    might be as hard as solving \cref{minfeas}.}

	\subsection{Bounds on Suboptimality, Infeasibility, and Distance to the Solution Set}

	Suppose we solve either \cref{minobj} or \cref{minfeas}
	to obtain a primal estimate $X=\objx$ or $X = \infeasx$.
	How can we estimate the suboptimality, infeasibility,
	and distance of $X$ to the solution set of~\cref{p}?

	The first two metrics are straightforward to compute.
	We can bound the suboptimality by $\epsilon_p (X)\leq \tr(CX)  - b^\top y$.
	We can compute the infeasibility as $\delta_p(X) = \twonorm{\mathcal{A}X-b}$.
	In the optimization literature,
	scaled versions of the suboptimality and infeasibility called KKT residuals \cite{wen2010alternating, malick2009regularization, zhao2010newton} are generally used as stopping criteria.

	The distance to the solution requires additional assumptions,
	such as surjectivity of the restricted constraint map $\mathcal A_V$.
	With these assumptions, \cref{lemma: minfeas1} yields a computable (but possibly loose) bound.
  We refer the interested reader to \Cref{sec: computablebound}.

	\section{Computational Aspects of the Dual SDP \cref{d}}\label{sec: sd}

	The previous two subsections showed how to efficiently
	recover
	an approximate primal solution from an approximate dual solution.
	We now discuss how to (approximately) solve the dual SDP \cref{d}
	with optimal storage and with a low per-iteration computational cost.
	Together, the (storage-optimal) dual solver and
	(storage-optimal) primal recovery
	compose a new algorithm for solving regular SDPs with optimal storage.

	\subsection{Exact Penalty Formulation}\label{sec: equivalent formulation}
	It will be useful to introduce an unconstrained version	of the dual SDP
	\cref{d},
	parametrized by a real positive number $\alpha$,
	which we call the penalized dual SDP:
	\beq \label{problem: dp}
	\ba{ll}
	\mbox{maximize} & b^\top y+ \alpha \min\{ \lambda_{\min}(C-\mathcal{A}^\top y),0\}.
	\ea
	\eeq
	That is, we penalize vectors $y$ that violate
	the dual constraint $C-\mathcal{A}^\top y\succeq 0$.

	Problem \cref{problem: dp} is an exact penalty formulation
	for the dual SDP \cref{d}.
	Indeed, the following lemma shows that the solution of Problem \cref{problem:
		dp}
	and the solution set of the dual SDP \cref{d} are the same when
	$\alpha$ is large enough. The proof is based on \cite[Theorem
	7.21]{ruszczynski2006nonlinear}.

	\begin{lemma} \label{lemma: exact}
Instate the assumptions in \Cref{sec:assumptions}.	If $b \not = 0$ and $ \alpha > 		\nucnorm{\xsol}$,
		then the penalized dual SDP \cref{problem: dp}
		and the dual SDP \cref{d} have the same solution $\ysol$.
	\end{lemma}
\begin{proof}[Proof of \cref{lemma: exact}]
	We first note that the dual solution $\ysol$ is the only 
	solution to 
$\min_{\lambda_{\max}(\Amap ^\top y-C)\leq 0} -b^\top y.$
	Using \cite[Theorem
	7.21]{ruszczynski2006nonlinear}, we know that the penalty form \cref{problem: dp} has
	$\ysol$ as its only solution as long as $\alpha>\alpha_0$ for any $\alpha_0\geq 0$
	satisfying the KKT condition:
	\[
	b \in\alpha_0 \Amap \left(\partial (\lambda_{\max}(-Z(\ysol)))\right),\quad \text{and}\quad \alpha_0 \lambda_{\max}(-Z(\ysol))=0.
	\]
	This is the case by noting
	$
	\xsol\in \tr(\xsol)\partial (\lambda_{\max}(-Z(\ysol))
	$,
	$\tr(\xsol) \lambda_{\max}(-Z(\ysol))=0$,
	and
	$
	\Amap(\xsol)=b
	$.
	Hence we can choose $\alpha_0 = \tr(\xsol)$.
\end{proof}
	Thus, as long as we know an upper bound on
	the nuclear norm of the primal solution,
	then we can solve Problem \cref{problem: dp} to find the dual
	optimal solution $\ysol$.
	%
	It is often easy to find a bound on $\nucnorm{\xsol}$ in the following two situations:
	\begin{enumerate}
		\item \emph{Nuclear norm objective.}
		Suppose the objective in \cref{p}  is $\nucnorm{X} = \tr(X)$.
		Problems using this objective include
		matrix completion \cite{candes2009exact},
		phase retrieval \cite{chai2010array},
		and covariance estimation \cite{chen2015exact}.
		In these settings,
		it is generally easy to find a feasible solution or
		to bound the objective via a spectral method.
		(See \cite{keshavan2010matrix} for matrix completion and
		\cite{candes2015phaserev} for phase retrieval.)
		\item \emph{Constant trace constraints.}
		Suppose the constraint $\mathcal{A}X =b$
		enforces $\tr(X) =\beta$ for some constant $\beta$.
		Problems with this constraint
		include Max-Cut \cite{goemans1995improved},
		Community Detection ~\cite{mathieu2010correlation},
		and
		PhaseCut in \cite{waldspurger2015phase}.
		Then any $\alpha >\beta$ serves as an upper bound.
		In the Powerflow \cite{bai2008semidefinite, madani2015convex} problems, we have
		constraints: $X_{ii}\leq \beta_i,\forall	i$.
		Then any $\alpha >\sum_{i=1}^n\beta_i$ serves as an upper bound.
		(The Powerflow problem does not directly fit into our standard form \cref{p},
		but a small modification of our framework can handle the problem.)
	\end{enumerate}

	When no such bound is available, we may search over $\alpha$ numerically.
	For example, solve Problem \cref{problem: dp} for $\alpha =2,4,8,\dots, 2^d$
	for some integer $d$ (perhaps, in parallel, simultaneously).
	Since any feasible $y$ for the dual SDP \cref{d}
	may be used to recover the primal, using \cref{minfeas} and \cref{minfeas},
	we can use any approximate solution of the penalized dual SDP,
	Problem \cref{problem: dp}, for any $\alpha$,
	as long as it is feasible for the dual SDP.

	Alternatively, the method in \cite{renegar2014efficient}
 which solves \cref{d} directly can also be utilized if a strictly dual feasible point is known.
	For example, $C\succ 0$ and $0\in \reals^\cons$ is a strictly dual feasible point.

	\begin{algorithm}[t]
		\caption{Dual Algorithm $+$ Primal Recovery} \label{alg: accelgrad+minfeas}
		\begin{algorithmic}[1]
			\REQUIRE{Problem data $\mathcal{A}$, $C$ and $b$}
			\REQUIRE{Positive integer $r$ and an iterative algorithm $\alg$ for solving
				the dual SDP }
			\FOR{$k=1,2,\dots$}
			\STATE Compute the $k$-th dual $y_k$ via $k$-th iteration of $\alg$
			\STATE Compute a recovered primal $\hat{X}_k =V\hat{S}V^\top$ using  Primal
			Recovery, \Cref{primalrecovery}.
			\ENDFOR
		\end{algorithmic}
	\end{algorithm}

	\subsection{Computational Cost and Convergence Rate for Primal Approximation}
		\label{sec: computationalcomplexityandconvergencerate}

	Suppose we have an iterative algorithm $\alg$ to solve the dual problem.
	Denote by $y_k$ the $k$th iterate of $\alg$.
	Each dual iterate $y_k$
	generates a corresponding primal iterate
	using either \cref{minfeas} or \cref{minfeas}.
	We summarize this approach to solving the primal SDP in
	\Cref{alg: accelgrad+minfeas}.

	The primal iterates $X_k$ generated by \cref{alg: accelgrad+minfeas}
	converge to a solution of the primal SDP~\cref{p} by our theory.%
	\footnote{Iterative algorithms for solving
		the dual SDP \cref{d} may not give a feasible point $y$.
		If a strictly feasible point is available,
		we can use the method of \cref{lm: boudningdualinfeasibility} or \cref{lm: averaging} in the appendix
		to obtain a sequence of feasible points
		from a sequence of (possibly infeasible) iterates
		without affecting the convergence rate.
		Alternatively, our theory can be extended to handle the infeasible case;
		we omit this analysis for simplicity.
		}
	However, it would be computational folly
	to recover the primal at every iteration:
	the primal recovery problem is much more computationally challenging
	than a single iteration of most methods for solving the dual.
	Hence, to determine when (or how often)
	to recover the primal iterate from the dual,
	we would like to understand
	how quickly the recovered primal iterates converge to the solution of
	the primal problem.

	To simplify exposition as we discuss algorithms for solving the dual,
	we reformulate the penalized dual SDP as a convex minimization problem,
	\beq \label{dpmin}
	\ba{ll}
	\mbox{minimize} & \objp(y):\;=-b^\top y+ \alpha \max\{
	\lambda_{\max}(-C+\mathcal{A}^\top y),0\},\\
	\ea
	\eeq
	which has the same solution set as the penalized dual SDP \cref{problem:
		dp}.

	We focus on the convergence of suboptimality and infeasibility,
	as these two quantities are easier to measure
	than distance to the solution set.
	Recall from Table~\ref{tb: comparison} that
	\begin{equation}\label{relation}
		\begin{aligned}
			\epsilon\text{-optimal dual feasible}\;y \xrightarrow[\text{or
				\minfeas}]{\text{\minobj}}
			(\mathcal{O}(\sqrt{\epsilon}),\mathcal{O}(\sqrt{\epsilon}))\text{-primal
				solution}\; X
		\end{aligned}
	\end{equation}
	if $\kappa = \bigO{1}$.
	Thus the convergence rate of the primal sequence
	depends strongly on the convergence rate of the algorithm
	we use to solve the penalized dual SDP.

	\subsubsection{Subgradient Methods, Storage Cost, and Per-Iteration Time Cost}

	We focus on subgradient-type methods for solving the penalized dual SDP
	\cref{problem: dp},
	because the objective $g_\alpha$ is nonsmooth but has an efficiently computable
	subgradient.
	Any subgradient method follows a recurrence of the form
	\beq
	\ba{ll}\label{e1}
	y_0 \in \reals^m \quad\text{and}\quad y_{k+1} = y_k - \eta_k g_k,
	\ea
	\eeq
	where $g_k$ is a subgradient of $g_\alpha$ at $y_k$ and
	$\eta_k\geq 0$ is the step size.
	Subgradient-type methods differ in
	the methods for choosing the step size $\eta_k$
	and in their use of parallelism.
	However, they are all easy to run for our problem
	because it is easy to compute a subgradient of the dual objective with penalty $g_\alpha$:
	\begin{lemma}
		Let $Z = C-\mathcal{A}^\top y$.
		The subdifferential of the function $g_\alpha$ is
		$$ \partial g_\alpha (y) = \begin{cases}-b+
		\conv \{\alpha\mathcal{A} (vv^\top)\mid  Zv =\lambda_{\min}(Z) v\}, &
		\lambda_{\min}(Z)<0\\
		-b, &\lambda_{\min}(Z)>0\\
		-b + \beta \conv \{ \alpha\mathcal{A} (vv^\top)\mid  Zv =\lambda_{\min}(Z) v,
		\beta\in [0,1]\}, &  \lambda_{\min}(Z)=0
		\end{cases}.$$
	\end{lemma}
	This result follows directly via standard subdifferential calculus
	from the subdifferential of the maximum eigenvalue $\lambda_{\max}(\cdot)$. Thus
	our storage cost is simply $\bigO{m+n}$ where $m$ is due to storing the
	decision variable $y$ and the gradient $g_k$, and $n$ is due to the intermediate
	eigenvector $v\in \reals^{n}$.
	The main computational cost in computing a subgradient
	of the objective in~\cref{dpmin}
	is computing the smallest eigenvalue $\lambda_{\min} (C-\mathcal{A}^\top y)$ and
	the corresponding eigenvector $v$ of the matrix $C-\mathcal{A}^\top y$.
	Since $C-\mathcal{A}^\top y$ can be efficiently applied to vectors
	(using the data access oracles~\cref{eq: efficientOperation}),
	we can compute this eigenpair efficiently using the randomized Lanczos method~\cite{kuczynski1992estimating}.

	\subsubsection{Convergence Rate of the Dual and Primal}

	The best available subgradient method \cite{johnstone2017faster}
	has convergence rate $\mathcal{O}(1/\epsilon)$ when
	the quadratic growth condition is satisfied.
	(This result does not seem to appear in the literature for SDP; 
	however, it is a simple consequence of
  \cite[Table 1]{johnstone2017faster}
  together with the quadratic growth condition proved in \cref{lemma: qg}.)
	Thus, our primal recovery algorithm has convergence rate
	$\mathcal{O}(1/\sqrt{\epsilon})$, using the relation between dual
	convergence and primal convergence in \cref{relation}.
	Unfortunately, the algorithm in \cite{johnstone2017faster}
	involves many unknown constants.
	In practice, we recommend using dual solvers that require
	less tuning such as  AcceleGrad \cite{levy2018online} which
	is the one we used in \Cref{sec: numerics}.

	\section{Numerical Experiments}\label{sec: numerics}

	In this section, we give a numerical demonstration of our approach
	to solving~\cref{p} via approximate complementarity.
	We first show that \Cref{primalrecovery} (Primal Recovery)
	recovers an approximate primal given an approximate dual solution.
	Next, we show that \Cref{alg: accelgrad+minfeas}
 	with primal recovery achieves reasonable accuracy $(10^{-1}\sim 10^{-2})$ for extremely large scale problems, \eg $10^5\times 10^5$,
	with substantially lower storage requirements compared to other SDP solvers.

	We test our methods on the Max-Cut and Matrix Completion SDPs,
	defined in \Cref{tb:numericproblem}.
	\begin{table}[tbhp]{
			\vspace{-10pt}
		\caption{Problems for numerics}\label{tb:numericproblem}
	\begin{center}
		\begin{tabular}{|c|c|}
			\hline
				\textbf{Max-Cut} 	 &  \textbf{Matrix Completion}\\
			\hline
			$
			\displaystyle
			\ba{ll}
			\mbox{minimize} & \tr( -LX) \\
			\mbox{subject to} & \diag(X) =  \ones \\
					              & X\succeq 0
			\ea
			$
			&
			$
			\displaystyle
			\ba{ll}
			\mbox{minimize} & \tr(W_1)+\tr(W_2) \\
			\mbox{subject to} & X_{ij} = \trux_{ij},\, (i,j)\in \Omega \\
					              & \begin{bmatrix}
						 						W_1 & X \\ X^\top & W_2
						 					\end{bmatrix}\succeq 0
			\ea
			$
			\\
			\hline
		\end{tabular}
	\end{center}}
\end{table}
For Max-Cut, $L$ is the Laplacian of a given graph. For Matrix Completion,
$\Omega$ is the set of indices of the observed entries of the underlying matrix
$\trux \in \reals^{\dm_1\times \dm_2}$. We use the dual penalty form 
\cref{dpmin} which defines defines $g_\alpha$ to measure both dual suboptimality and infeasibility.
We set $\alpha =1.1n$ for Max-Cut and $2.2\times\nucnorm{\bar{X}}$ for matrix completion throughout our experiments.

	\subsection{Primal Recovery}\label{sec: primalRecovery}
	Our first experiment confirms numerically that \Cref{primalrecovery} (Primal Recovery)
	recovers an approximate primal from an approximate dual solution,
	validating our theoretical results.
	As an example, we present results for the Max-Cut SDP using a Laplacian $L\in \reals^{800\times 800}$
	from the G1 dataset of \cite{Gset}.
	Results for matrix completion and for other Max-Cut problems are similar;
	we present the corresponding experiment for matrix completion in \Cref{sec: primalrecoveryMatrixCompletion}.
	To evaluate our method, we compare the recovered primal
	with the primal dual solution $\xsol$, $\ysol$
	obtained with Sedumi, an interior point solver \cite{sturm1999using}.
	Empirically, the rank of the primal solution $\rsol = 13$.

	To obtain approximate dual solutions $y$,
	we perturb the true dual solution $\ysol$ to generate
	\[
	y = \ysol + \varepsilon s \twonorm{\ysol},
	\]
	where $\varepsilon$ is the noise level, which we vary from $1$ to $10^{-5}$,
	and $s$ is a uniformly random vector on the unit sphere in $\reals^m$.
	For each perturbed dual $y$ and for each rank estimate $r \in \{\rsol, 3\rsol\}$,
	we first solve \cref{minfeas} to obtain a solution $\infeasx$,
	and then solve \cref{minobj} with $\delta =1.1\twonorm{\mathcal{A}\infeasx-b}$.
	We measure the suboptimality of the perturbed dual using
	the relative suboptimality $\frac{|p^\star+g_\alpha(y)|}{|p^\star|}$.
	We measure the distance of the recovered primal to the solution in three ways:
	relative suboptimality $| \tr(CX) - \pval | / \pval$,
	relative infeasibility $\|\mathcal A X - b\| / \|b\|$,
	and relative distance to the solution set $\fronorm{X - \xsol} / \fronorm{\xsol}$.

	\Cref{fig:proofofconcept} shows distance of the recovered primal
	to the solution.
	The blue dots show the primal recovered using $r=\rsol$, while
	the red dots show the primal recovered using $r=3\rsol$.
	The blue and red curves are fit to the dots of the same color to provide a visual guide.
	The red line ($r=3\rsol$) generally lies below the blue line ($r=\rsol$),
	which confirms that larger ranks produce more accurate primal reconstructions.

	These plots show that the recovered primal approaches the true primal solution
	as the dual suboptimality approaches zero,
	as expected from our theory.\footnote{To be precise,
	the theory we present above in \Cref{thm: minfeas} and \Cref{thm: minobj}
	requires the approximate dual solution to be feasible,
	while $y$ may be infeasible in our experiments.
	An extension of our results can show similar bounds when
	$y$ is infeasible but $g_\alpha(y)$ is close to $-\dval$.}
	From \cref{tb: comparison}, recall that we expect
	the primal solution recovered from an $\epsilon$ suboptimal dual solution
	to converge to the true primal solution as $\bigO{\sqrt{\epsilon}}$ with respect to all three measures.
	The plots confirm this scaling for distance to solution and infeasibility,
	while suboptimality decays even faster than predicted by our theory. 
	By construction, the primal suboptimality of \cref{minobj} is smaller than
	that of \cref{minfeas};
	however, the plots measure primal suboptimality by its absolute value.
	The kink in the curves desribing primal suboptimality for \cref{minobj}
	separates suboptimal primal solutions (to the left) from superoptimal solutions (to the right).
	Finally, notice that $3\rsol =39$ is close to the Barvinok--Pataki bound.
	Interestingly, \cref{minfeas} still performs better with this large feasible set ($r = 3\rsol$)
	than with a smaller one ($r = \rsol$),
	although our theory does not apply.
	\begin{figure}[tbhp]
		\centering
		\begin{subfigure}[b]{1\textwidth}
			\includegraphics[width=1\linewidth]{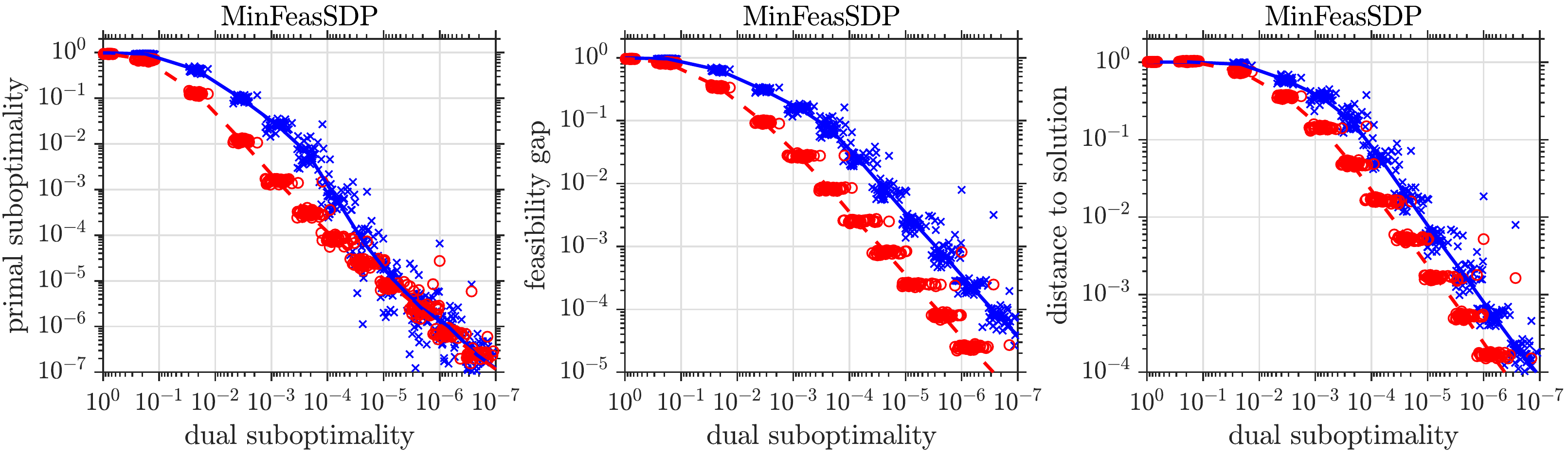}
			\caption{\cref{minfeas}}
			\label{Figure: minfeas}
		\end{subfigure}
		\begin{subfigure}[b]{1\textwidth}
			\includegraphics[width=1\linewidth]{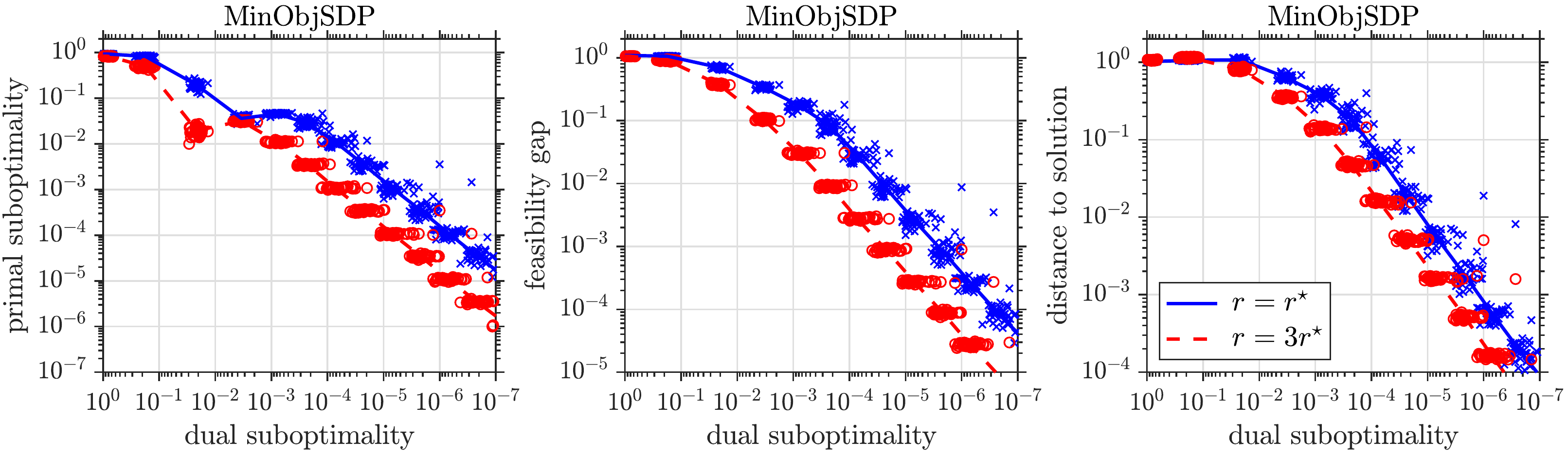}
			\caption{\cref{minobj}}
			\label{Figure: minobj}
		\end{subfigure}
		\caption[Max-Cut]{The plots shows the primal recovery performance of
			\cref{minfeas} (upper) and \cref{minobj} (lower) in terms of 
			(relative) primal suboptimality $| \tr(CX) - \pval | / \pval$,
			(relative) infeasibility gap $\|\mathcal A X - b\| / \|b\|$,
			and (relative) distance to solution $\fronorm{X - \xsol} / \fronorm{\xsol}$. The horizontal
			axis is (relative) dual  suboptimality $\frac{|p^\star+g_\alpha(y)|}{|p^\star|}$. The blue dots corresponds to the choice $r=\rsol$ and the red
			dots corresponds to the choice  $r=3\rsol$ in \cref{alg: accelgrad+minfeas}.}\label{fig:proofofconcept}
	\end{figure}

	\subsection{Storage efficiency comparing to existing solvers} \label{sec: storageVSdimension}
	Experiments in this section show that our method \cref{alg: accelgrad+minfeas} uses less storage
	(for high dimensional problems) than existing algorithms.
	We use AccelGrad \cite{levy2018online} as the dual solver.
	We solve \cref{minobj} to recover the primal, using $\gamma=1.1$ and several different rank estimates $r$.
	We compare \cref{alg: accelgrad+minfeas} against the mature SDP solvers
	Mosek \cite{mosek2010mosek}, SDPT3 \cite{toh1999sdpt3} and Sedumi \cite{sturm1999using},
	and the state-of-the-art SDP solver SDPNAL+ \cite{sun2020sdpnal+}.
	\cref{fig:storage versus dimension plot} (our method is labeled as CSSDP) shows how the storage required for these methods
	scales with the side length $\dm$ of the primal decision variable $X\in \sym^\dm$ for Max-Cut and matrix completion.
	Our Max-Cut problems are drawn from Gset and the DIMACS10 group \cite{DIMACS10}.
	Our matrix completion problems are simulated. 
	We generate rank 5 random matrices $\bar{X} = UV \in \reals^{n_1 \times n_2}$ where
	$U \in \{\pm 1\}^{n_1\times 5}$ and $V \in \{\pm 1\}^{5 \times n_2}$ are random sign matrices.
	We vary the dimensions by setting $n=75c$, $m=50c$ and varying $c= 1,2,4,8,\dots$
	The $25(n_1+n_2)\log(n_1+n_2)$ observations are chosen uniformly at random.

	As can be seen from the plots, the mature solvers Mosek, SDPT4, and Sedumi exceed the storage limit $16$GB
	for matrix completion when $n>10^3$ and for Max-Cut when $n>10^4$.
	SDPNAL+ uses less storage than the mature solvers.
	However, the storage still exceeds $16$GB when $n>10^4$ for both problems.
	In contrast, our method (labeled as CSSDP) scales linearly with the dimension (for any $r$),
	and can solve problems with $n=10^6$ on a 16GB laptop.

	\begin{figure}
		\centering
				\begin{subfigure}[b]{0.48\textwidth}
		\includegraphics[width=1\linewidth]{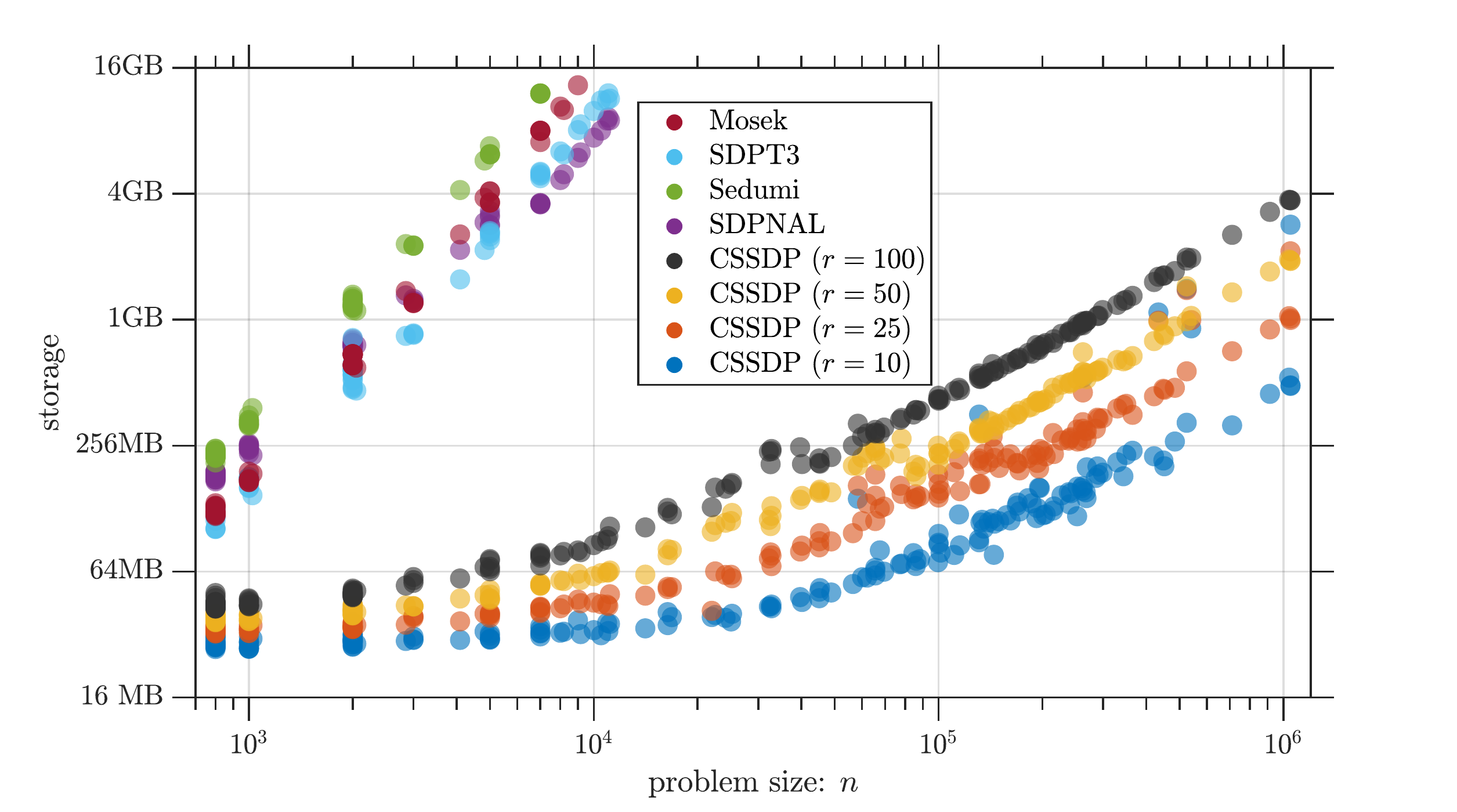}
		\caption{Max-Cut}
		\label{Figure:MaxCutStorageVesussize}
			\end{subfigure}
		\begin{subfigure}[b]{0.45\textwidth}
			\includegraphics[width=1\linewidth]{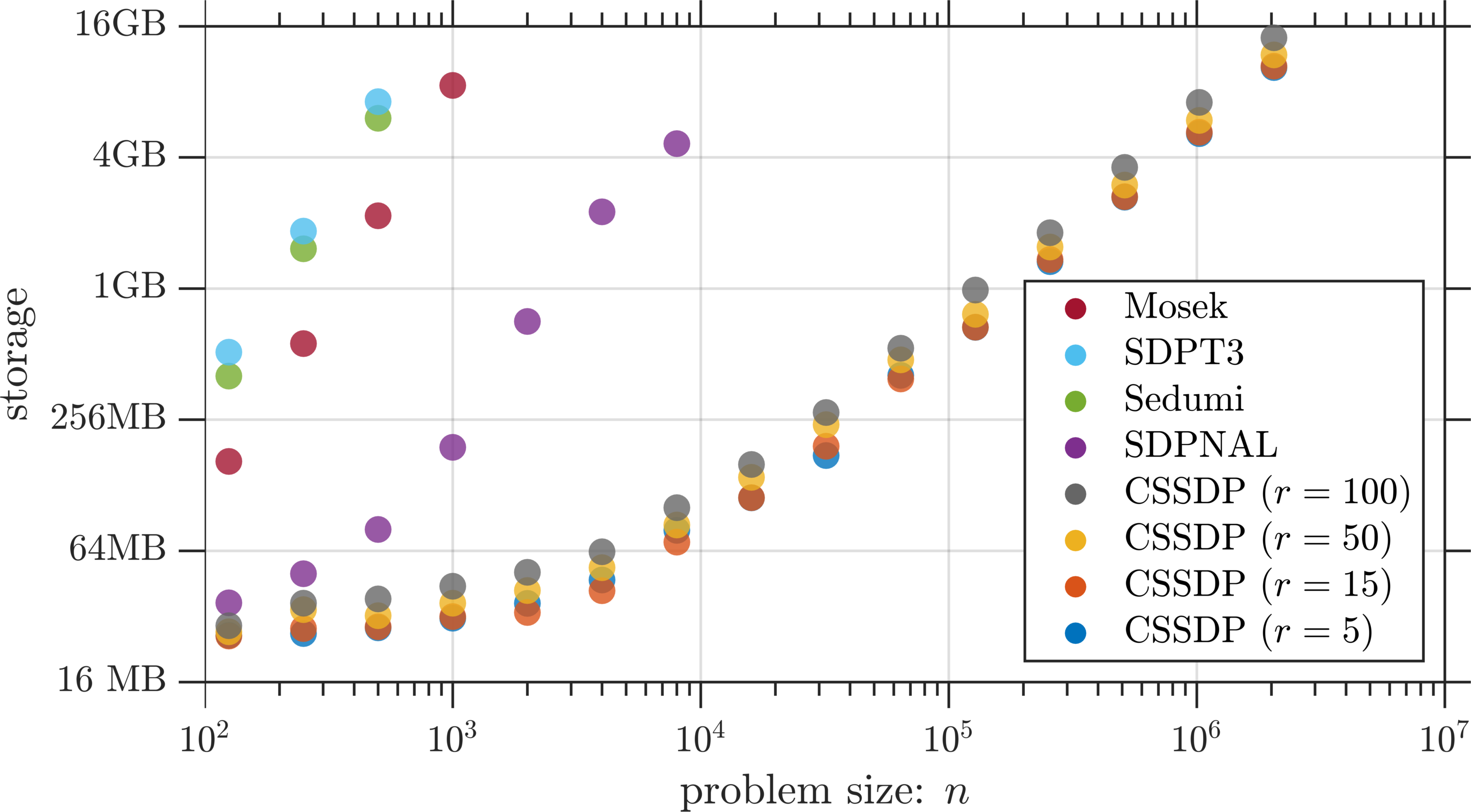}
			\caption{Matrix Completion}
			\label{Figure: matrixcompletionStorageVesussize}
		\end{subfigure}
		\caption[Storage versus dimension]{Here we compare our \cref{alg: accelgrad+minfeas} (shown as CSSDP with different choice
			of $r$), with Mosek, SDPT3, Sedumi, and SDPNAL+. Dots for Mosek, SDPT3, Sedumi, and SDPNAL+ are not shown for large $n$ because they exceed $16$GB.}\label{fig:storage versus dimension plot}
	\end{figure}
	\subsection{Accuracy versus time}\label{sec: accuracyversusTime}
	In this section, we demonstrate that \Cref{alg: accelgrad+minfeas} can solve large scale problems
	that are inaccesible to the SDP solvers	Mosek, SDPT3, Sedumi, and SDPNAL+ due to storage limit.
	Since primal recovery is substantially more expensive than a dual iteration (see \Cref{sec: computationalcomplexityandconvergencerate}),
	we recover the primal only at iterations $10,10^2,10^3,10^4,10^5,\dots$
	We solve both \cref{minfeas}) (option 1) and \cref{minobj}) with $\gamma =1.1$ (option 2) in \cref{primalrecovery}.
	These solutions are shown as the solid and dotted lines in \cref{fig:accuracy_versus_time_and_iteration_counter}, respectively.
	Since we do not know the optimal solution, we track performance using two DIMACS measures of
	(scaled) infeasibility and suboptimality, 
\begin{align*}
\textrm{relative feasibility gap:}\qquad &\frac{\twonorm{\Amap(X)-b}}{\twonorm{b}+1} \\
\textrm{relative primal-dual gap:}\qquad &\frac{|\tr(CX)+g_\alpha(y)|}{|\tr(CX)|+|g_\alpha (y)|+1}.
\end{align*}
These measures are commonly used to benchmark SDP solvers \cite{wen2010alternating, malick2009regularization, zhao2010newton}.
Here $\tr(CX)+g_\alpha(y)$ bounds primal suboptimality.
It is traditional to use $\tr(CX) - b^\top y$;
however, here $y$ is not necessarily dual feasible and so
this simpler measure does not bound primal suboptimality.
Results for the Max-Cut SDP on the smallworld graph in the DIMACS10 group \cite{DIMACS10},
with a decision variable of size $10^5\times 10^5$, are shown in \Cref{Figure:MaxCutTimeAccuracy}.
Results for a matrix completion problem, simulated as described in \Cref{sec: storageVSdimension} with $c=1000$,
with decision variable size $(n_1+n_2)^2=(1.25\times 10^{5})^2$ with $n_1=75000$ and $n_2 = 50000$, and over $3.6\times 10^7$ many 
constraints are shown in \Cref{Figure: MatrixCompletionAccuracy}.

As can be seen, the proposed method reaches a solution with $10^{-1}$ infeasibility and $10^{-3}$ suboptimality
in $10^4$ seconds when the rank parameter is large ($r=100$ or $r=250$) for Max-Cut and  $10^{-1}$ infeasibility and $10^{-3}$ suboptimality
in $10^5$ seconds when $r=5$ or $15$ for matrix completion.
These ranks are far smaller than the Barvinok--Pataki bound.
Again, \cref{minfeas} outperforms \cref{minobj} and is faster and easier to compute.
We plot points according to their (dual) iteration counter;
the top point on each line corresponds to dual iteration 10.
Primal recovery from accurate dual iterates is both more accurate and faster,
so primal iterates recovered from early dual iterates can be dominated by those recovered from later iterates.

Additional experiments can be found in \Cref{sec: additionalNumerics}.

	\begin{figure}
	\centering
	\begin{subfigure}[b]{1\textwidth}
		\includegraphics[width=1\linewidth]{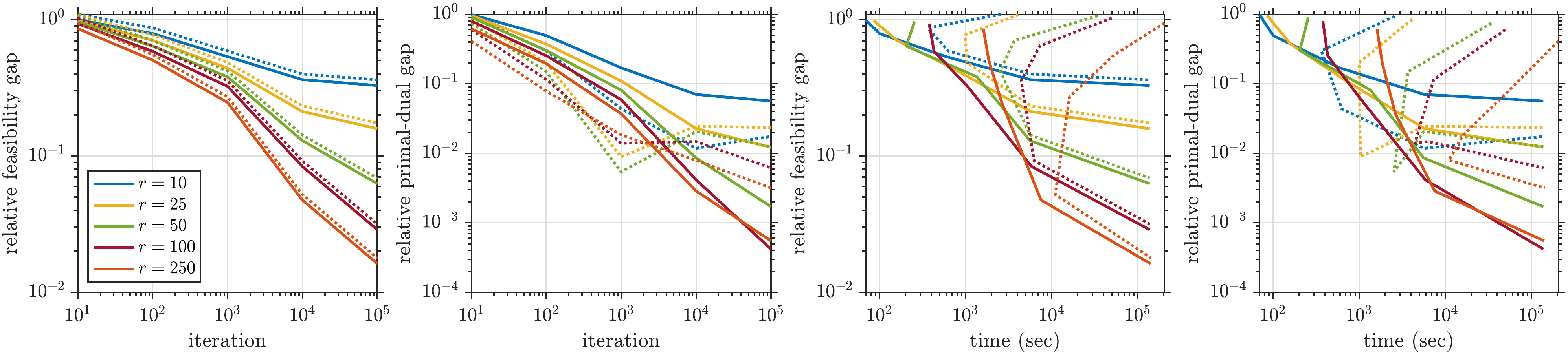}
		\caption{Max-Cut: smallworld}
		\label{Figure:MaxCutTimeAccuracy}
	\end{subfigure}
	\begin{subfigure}[b]{1\textwidth}
	\includegraphics[width=1\linewidth]{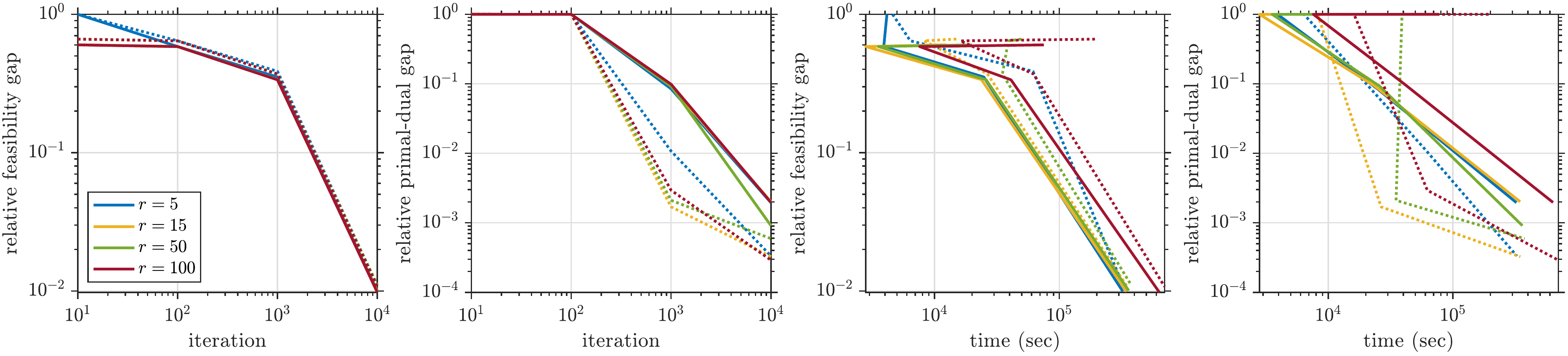}
	\caption{MatrixCompletion: $n_1 = 75000, \quad n_2 = 50000$}
	\label{Figure: MatrixCompletionAccuracy}
\end{subfigure}
	\caption[Accuracy versus actual time and iterations]{
	Convergence of \cref{alg: accelgrad+minfeas} with \cref{minfeas} as the solid line and \cref{minobj} as the dotted line.
	Primal recovery from accurate dual iterates is both more accurate and faster,
	so primal iterates recovered from early dual iterates can be dominated by those recovered from later iterates.}\label{fig:accuracy_versus_time_and_iteration_counter}
\end{figure}

	\section{Conclusions}
	\label{sec:conclusions}
	This paper presents a new theoretically justified method to
	recover an approximate solution to a primal SDP from
	an approximate solution to a dual SDP,
	using complementarity between the primal and dual optimal solutions.
	We present two concrete algorithms for primal recovery,
	which offer guarantees on the suboptimality, infeasibility,
	and distance to the solution set of the recovered primal under
	the regular conditions on the SDP, and we demonstrate that
	this primal recovery method works well in practice.

	We use this primal recovery method to develop the
	first storage-optimal algorithm to solve regular SDP:
	use any first-order algorithm to solve a penalized version of the dual problem,
	and recover a primal solution from the dual.
	This method requires $O(m+nr)$ storage:
	the storage is linear in the number of constraints $m$
	and in the side length $n$ of the SDP variable,
	when the target rank $r$ of the solution is fixed.
	These storage requirements improve on the storage requirements
	that guarantee convergence for nonconvex factored
	(Burer-Monteiro) methods to solve the SDP, which scale as $O(\sqrt{m}n)$.
	Furthermore, we show that no method can use less storage without
	a more restrictive data access model or
	a more restrictive representation of the solution.
	We demonstrate numerically that our algorithm is able to solve
	SDP of practical interest including Max-Cut and Matrix Completion.

The ideas illustrated in this paper can be extended to solve problems
with inequality constraints. We leave this extension for future work.

\appendix
\section{Lemmas for \Cref{sec: introduction}} \label{sec: lemmaOfIntroduction}
To establish \cref{lemma: qg}, we prove a lemma concerning the
operator $\Amap_{\vrepox}$.

\begin{lemma}\label{thm: uniqueness}
	Instate the hypothesis of \Cref{sec:assumptions}. Then
	$\nullspace(\mathcal{A}_{V^\star}) = \{0\}$.
\end{lemma}
\vspace{-8pt}
\begin{proof}[Proof of \cref{thm: uniqueness}]
	Suppose by way of contradiction that $\ker(\Av{\vrepox})\not= \{0\}$.
	Let $S \in \ker(\Av{\vrepox})$, so $\Amap_{\vrepox}(S) =0.$
	Recall $\xsol = \vrepox\ssol (\vrepox)^\top$ for some unique $\ssol \succ 0$.
	Hence for some $\alpha_0>0$, $\ssol +\alpha S\succeq 0$
	for all $|\alpha|\leq \alpha_0$.
	Now pick any $\alpha$ with $|\alpha|\leq \alpha_0$ to see
\vspace{-8pt}
\[
	\Amap
	(X_{\alpha}) = \Av{\vrepox}(\ssol +\alpha S) = \Av{\vrepox}(\ssol ) + 0 = b.
\vspace{-5pt}
\]
	This shows $X_{\alpha}$ is feasible for all $|\alpha|\leq \alpha_0$.
	But we can always find some $|\alpha|\leq \alpha_0$, $\alpha \ne 0$, so that
	$\tr(\objc X_\alpha) = \pval + \alpha\tr(\objc \vrepox (S\vrepox)^{\top})\leq \pval$.
	This contradicts the assumption that $\xsol$ is	unique.
	Hence we must have $\nullspace(\Av{\vrepox}) = \{0\}$.
\end{proof}
\begin{proof}[Proof of \cref{lemma: qg}]
 Consider the linear operator $\mathcal{D}$ defined in \cref{lemma: qg}.
	An argument similar	to the proof of \cref{thm: uniqueness} shows $\ker(\mathcal{D}) =\{0\}$ (see Lemma \ref{lem: dualuniquenessmaplemma} below
     for more details).
	Hence
\vspace{-10pt}
	\[
	\|(Z(y),y)-(\zsol,\ysol)\| \leq \frac{1}{\sigma_{\min}(\mathcal{D})}\|
	\mathcal{D}(Z(y)-\zsol,y-\ysol)\|.
	\vspace{-5pt}
	\]
By utilizing Lemma \ref{lem: MainProjectionDistanceInnerProduct} with $X=Z(y)$ and $Z=\xsol$ and
note $\epsilon=\tr(Z(y)\xsol)=\epsilon_d(y)=b^\top \ysol -b^\top y$ (from strong duality), we see that
\vspace{-10pt}
	\[
	\fronorm{Z(y)- (U^\star)(U^\star)^\top Z(y) (U^\star)(U^\star)^\top}\leq
\frac{\epsilon}{\lambda_{\min>0}(\xsol)}+
	\sqrt{\frac{2\epsilon}{\lambda_{\min>0}(\xsol)}\opnorm{Z(y)}}.
	\]
We also have
\vspace{-10pt}
\[
\mathcal{D}(Z(y)-\zsol,y-\ysol) = (\mathcal{D}Z(y), 0) = Z(y) -
	(U^\star)(U^\star)^\top Z(y) (U^\star)(U^\star)^\top.
	\vspace{-5pt}
\]
	Combining the above pieces, we get the results in \cref{lemma: qg}.
\end{proof}
\begin{lemma}\label{lem: dualuniquenessmaplemma}
	Instate the hypothesis of \Cref{sec:assumptions}. Then
	$\nullspace(\mathcal{D}) = \{0\}$.
\end{lemma}
\begin{proof}
	Suppose $\mathcal{D}\not=\{0\}$, then there is some $\hat{Z}$ and $\hat{y}$ such that \[
	\mathcal{D}(\hat{Z},\hat{y})=0 \implies \hat{Z}=-\Amap^\top \hat{y} \quad \text{and}\quad \hat{Z}= (\urepox \urepox^\top) \hat{Z} (\urepox \urepox^\top)
	\]
	Then we claim $y_\gamma =\ysol +\gamma \hat{y}$ is also an solution to the dual SDP \cref{d}, which violates 
	the unique assumption. Indeed, $Z(y_\gamma)$ satisfies 
	\begin{equation}
	\begin{aligned} 
	Z(y_\gamma) & = C -\Amap^\top (y_\gamma)\\
	&=\zsol - \gamma \Amap ^\top\hat{y}\\ 
	& \overset{(a)}{=} \zsol +\gamma \hat{Z} \\
	& \overset{(b)}{=} \urepox \left(\urepox ^\top \zsol \urepox +\gamma \urepox ^\top \hat{Z}\urepox   \right) \urepox  ^\top
	\end{aligned} 
	\end{equation}
	Here we use the $\mathcal{D}(\hat{Z},\hat{y})=0$ in step $(a)$ and $(b)$. Because of strict complementarity, $\range(\urepox) = \range(\zsol)$
	and hence $ \urepox ^\top \zsol \urepox \succ 0$. Hence there is some constant $c>0$ such that for all 
	$|\gamma|<c$, $Z(y_\gamma)$ is still feasible for the dual SDP \cref{d}. But the 
	objective  $b^\top \ysol + \gamma b^\top \hat{y}$ then can be larger to $b^\top \ysol$ for all $|\gamma|<c$ and 
	$\gamma$ matching the sign of $b^\top \hat{y}$, and equal to $b^\top \ysol$ for all $|\gamma|<c$ if $b^\top \hat{y}=0$.
\end{proof} 

\section{Lemmas from \Cref{sec: oracleinequalities}}\label{sec: lemmasForSectionOfOracleinequalities}
We first prove \Cref{lem: MainProjectionDistanceInnerProduct} concerning the distance to subspaces and the
inner product.
\begin{proof}[Proof of \cref{lem: MainProjectionDistanceInnerProduct}]
	Complete $V$ to form a basis $W =[U\,V]$ for $\reals^n$,
	where $U=[v_{r+1},\dots ,v_n]\in \reals^{n\times (n-r)}$ and where
	$v_{i}$ is the eigenvector of $Z$ associated with the $i$-th
	smallest eigenvalue.

	Rotating into the coordinate system formed by $W=[V ;U]$, let us compare
	$X$ and its projection into the space spanned by $V$, $P_V(X)\defn  VV^\top XVV^\top $,
	$$W^\top X W^  =
	\begin{bmatrix}
	U^\top X U & U^\top X V\\
	V^\top X U & V^\top X V
	\end{bmatrix}, \quad \text{and}\quad W^\top P_V(X) W =  \begin{bmatrix}
	0 & 0\\
	0 & V^\top X V
	\end{bmatrix}.$$
	Let $X_1 = U^\top X U$, $B =U^\top X V$ and $X_2 = V^\top X V$.
	Using the unitary invariance of $\fronorm{\cdot}$, we have
	$P_{V^\perp }(X)\defn X - VV^\top X VV^\top$ satisfying 
	\begin{equation}\label{eq: MainLemmaEq0}
	\begin{aligned}\fronorm{ P_{V^\perp }(X)}=
	\fronorm{ W^\top XW - W^\top VV^\top X VV^\top W }= \fronorm{\begin{bmatrix}
		X_1 & B\\
		B &  0
		\end{bmatrix}}.
	\end{aligned}
	\end{equation}
	A similar equality holds for $\nucnorm{\cdot}$.
	Thus we need only bound the terms $X_1$ and $B$.
	Applying \Cref{lemma: auxSecOracleinequalities} to
	$WX
	W^\top= \begin{bmatrix}
	X_1 & B \\ B^\top  & X_2
	\end{bmatrix}$, we have
	\beq
	\ba{ll}\label{eq: MainLemmaEq1}
	\opnorm{X_2} \tr(X_1) & \geq \nucnorm{BB^\top} .\\
	\ea
	\eeq

	Since all the vectors in $U$ have corresponding eigenvalues
	at least as large as the threshold $T=\lambda _{n-r}(Z)>0$,
	and $Z\succeq 0$ by assumption, we have
	\begin{align}
	\epsilon =\tr(ZX)
	= \sum_{i=1}^n\lambda_{n-i+1}(Z) v_i^\top X v_i
	\geq \lambda_{n-r}(Z) \sum_{i=r+1}^n v_i^\top X v_i.
	\end{align}
	This inequality allows us to bound $\fronorm{X_1}$ as
	\begin{align} \label{eq: MainLemmaEq2}
	\dfrac{\epsilon}{T} \geq
	\sum_{i=r+1}^n v_i^\top X v_i
	= \tr(UXU^\top)
	=\tr(X_1)
	= \nucnorm{X_1}
	\geq \fronorm{X_1},
	\end{align}
	where we recall $X_1\succeq 0$ to obtain the second to last equality.
	Combining \cref{eq: MainLemmaEq1}, \cref{eq: MainLemmaEq2}, and $\opnorm{X_2}\leq
	\opnorm{X}$, we  have
	\begin{align} \label{eq: MainLemmaEq3}
	\nucnorm{BB^\top} \leq  \frac{\epsilon}{T}
	\opnorm{X_2} \leq\frac{\epsilon}{T}\opnorm{X}.
	\end{align}
	Basic linear algebra shows
	\beq
	\ba{ll}\label{eq: MainLemmaEq4}
	\fronorm{\begin{bmatrix}
			0 & B\\
			B^\top  &  0
	\end{bmatrix}}^2 =  \tr\biggr(\begin{bmatrix}
		BB^\top  & 0\\
		0 &  B^\top B
	\end{bmatrix}\biggr)\leq2\tr(BB^\top )   = 2\nucnorm{BB^\top}.
	\ea
	\eeq
	Combining pieces, we bound the error in the Frobenius norm:
	\begin{equation}\label{eq: fronormbound}
	\begin{aligned}
	\fronorm{X - V V^\top X V V^\top }& \ineqov{(a)}
	\fronorm{X_1} + \fronorm{\begin{bmatrix}
		0& B\\
		B^\top  &  0
		\end{bmatrix}}
	&\ineqov{(b)}  \frac{\epsilon}{T} + \sqrt{2\nucnorm{BB^\top} }\\
	&  \ineqov{(c)} \frac{\epsilon}{T} + \sqrt{\frac{2\epsilon}{T}\opnorm{X}},
	\end{aligned}
	\end{equation}
	where step $(a)$ uses \cref{eq: MainLemmaEq0} and the triangle
	inequality; step $(b)$ uses \cref{eq: MainLemmaEq1} and \cref{eq: MainLemmaEq4};
	and step $(c)$ uses \cref{eq: MainLemmaEq3}.
	Similarly, we may bound the error in the nuclear norm:
	\begin{align*}
	\nucnorm{\xsol - VV^\top \xsol VV^\top}
	& \ineqov{(a)}  \nucnorm{X_1}+ \nucnorm{\begin{bmatrix}
		0 & B\\
		B^\top  &  0
		\end{bmatrix}}
	& \ineqov{(b)} \tr(X_1)+\sqrt{2 r}\fronorm{\begin{bmatrix}
		0 & B\\
		B^\top  &  0
		\end{bmatrix}}\\
	&  \ineqov{(c)}\frac{\epsilon}{T} + 2\sqrt{\frac{
			r\epsilon}{T}\opnorm{\xsol}}.
	\end{align*}
	Step $(a)$ follows step $(a)$ in \cref{eq: fronormbound}. Step $(b)$ uses the fact that
	$\begin{bmatrix}
	0 & B\\
	B^\top  &  0
	\end{bmatrix}$ has rank at most $2r$. Step $(c)$ follows the step $(b)$ and $(c)$ in the inequality
	\cref{eq: fronormbound}.
\end{proof}

\begin{lemma} \label{lemma: auxSecOracleinequalities}
	Suppose
	$Y = \begin{bmatrix}
	A & B \\ B^\top  & D
	\end{bmatrix} \succeq 0$.
	Then $ \opnorm{A} \tr(D) \geq \nucnorm{BB^\top}.$
\end{lemma}
\begin{proof}
	For any $\epsilon>0$, denote  $A_\epsilon = A + \varepsilon I$ and
	$ Y_\epsilon = \begin{bmatrix}
	A_\epsilon & B \\ B^\top  & D
	\end{bmatrix}.$  We know $Y_\epsilon$ is psd,
	as is its Schur complement
	$D - B^\top  A_\epsilon ^{-1}B \succeq 0$ with trace $\tr(D) - \tr({A}_\epsilon^{-1} BB^\top ) \geq 0.$
Von Neumann's trace inequality \cite{mirsky1975trace} for $A_\epsilon$, $BB^{\top}\succeq 0$
	shows $\tr(A_\epsilon^{-1} BB^\top ) \geq \dfrac{1}{\opnorm{A_{\epsilon}}}\nucnorm{BB^\top}$.
	Use this with $\tr(D) - \tr({A}_\epsilon^{-1} BB^\top ) \geq 0$ to see
	$\tr(D) \geq \frac{1}{\opnorm{A_\epsilon}}\nucnorm{BB^\top}.$
	Multiply by $\opnorm{A_{\epsilon}}$ and let $\varepsilon \to 0 $
	to complete the proof. 
\end{proof}
\section{Chambolle-Pock for \ref{minobj}}\label{sec: cpminobj}
Here we state how to use Chambollo-Pock to solve \cref{minobj}:
\beq  \tag{\minobj}
\ba{ll}
\mbox{minimize} & \tr( C_V S)\\
\mbox{subject to} & \|\Av{V}(S) -  b\|\leq \delta
\quad\text{and}\quad S\succeq 0. \\
\ea
\eeq
In Chambollo-Pock, we have iterates $S_k\in \sym^{r}$, $\bar{S}_k\in \sym^r$ and $y_k\in \reals^{m}$. Denote the projection 
to the $\delta$-radius ball $\mathbf{B}_{\delta}\{y\in \reals^\cons \mid \twonorm{y}\leq \delta \}$ 
as $P_{\tiny \mathbf{B}_\delta}$ and the projection to $\sym_+^r$ as $P_{\tiny \sym_+^r}$. We choose 
$\tau,\sigma$, and $\theta>0$, and start at  $S_0\in \sym_+$. The iteration scheme is as follows: 
\begin{align}
	y_{k+1} & = y_k +\sigma \left(\Av{V}(\bar{S}_k)-b\right)-\sigma P_{\tiny \mathbf{B}_\delta } \left( \frac{1}{\sigma}y_k+\Av{V}\bar{S}_k-b\right),\\ 
	S_{k+1} & = P_{\sym_+^r}\left(S_k -\tau \Av{V}^\top y_{k+1} -\tau C_V\right) ,\\ 
	\bar{S}_{k+1} & = S_{k+1} + \theta \left(S^{k+1}-S^k\right).
\end{align}
We can compute $C_V$ in $r^2
L_{C}$ time before the iteration scheme and then store it using $r^2$ storage.  Each iteration only requires one call to each $\Av{V},\Av{V}^\top$,  $P_{\tiny \mathbf{B}_\delta},$ and $ P_{\tiny \sym_+^r}$, and a constant number of addition in $\reals^m$ and $\sym_+^r$, the per iteration flop counts is $\bigO{r^2 L_{\mathcal{A}} + rL_{\mathcal{A}^\top}
	+ m + r^2n}$. We need to store residues $\Av{V}(S)-b$, one dual iterate  of size $m$, two primal iterates of size $r^2$ and a few intermediate quantities such as $\Av{V}^\top y_{k+1} $ and $\frac{1}{\sigma}y_k+\Av{V}\bar{S}_k-b$, which requires $\bigO{r^2+m}$ storage. Hence the method is indeed storage-optimal.

\section{Computable bounds of the distance to solution}\label{sec: computablebound}
We described a way of computing the distance to the solution here, given a 
bound on $\opnorm{\xsol}$ and $\sigma_{\max}(\Amap)$. We note the assumptions here are weaker.
\begin{theorem}[Computable Bounds]\label{thm: computablebound}
	Suppose \eqref{p} and \eqref{d} admit solutions
	and satisfy strong duality, \Cref{eqn:strong-duality}.
	Let $y \in \reals^m$ be a dual feasible point
	with suboptimality $\epsilon = \epsilon_d(y) = b^\top \ysol - b^\top y$.
	For a positive integer $r$, form the orthonormal matrix $V \in \reals^{n \times r}$,
	as in \cref{primalrecovery}, and compute the threshold $T = \lambda_{n-r}(Z(y))$.
	
	If $\sigma_{\min} (\mathcal{A}_V)>0$ and $T>0$ 
	and $\opnorm{\xsol }\leq B$ for some solution $\xsol$ to \eqref{p}, 
	then 
	\begin{align}\label{explicitcomputablebound}
		\fronorm{\infeasx - \xsol } \leq\left (1+ \frac{\sigma_{\max}(\mathcal{A})}{\sigma_{\min}(\mathcal{A}_V)}\right)
		\bigg( \frac{\epsilon}{T} +
		\sqrt{2\frac{\epsilon}{T} B}\bigg). 
	\end{align}
	
	Moreover, any solution $\tilde{S}$ of \cref{minobj} with infeasibility parameter
	\begin{align}\label{deltachoice}
		\delta
		\geq \delta_0
		:=\sigma_{\max}(\mathcal{A})\left(\frac{\epsilon}{T}+2\sqrt{\frac{2\epsilon
				B}{T}}\right)
	\end{align}
	leads to an $(\epsilon_0,\delta)$-solution $\objx$ 
	for the primal SDP \cref{p} with
	\begin{align}\label{epsilonchoice}
		\epsilon_0 =   \min\biggr\{
		\fronorm{C} \biggr(\frac{\epsilon}{T} + \sqrt{2\frac{\epsilon}{T} B}\biggr)
		, \opnorm{C} \biggr(\frac{\epsilon}{T} + \sqrt{2\frac{r \epsilon}{T} B}\biggr)
		\biggr\}.
	\end{align}
\end{theorem}

\begin{proof}
	The inequality \eqref{explicitcomputablebound} is a direct application of \cref{lemma: minfeas0}
	and \cref{lemma: minfeas1}. The bound on $\delta_0$ and $\epsilon_0$ follows the same proof as in \cref{thm: minobj}.
\end{proof}

Since $\opnorm{\xsol}\in[\frac{1}{\rsol} \nucnorm{\xsol}, \nucnorm{\xsol}]$, 
we might use $\nucnorm{\xsol}$ as a substitute of the operator norm. A bound on 
$\nucnorm{\xsol}$ is often available, see
\Cref{sec: equivalent formulation}.
we can use \cref{lemma: minfeas1} to bound the distance to the
solution for \cref{minfeas}.
Moreover, based on this bound,
we can also estimate $\epsilon_p,\delta_p$ for
the solution $\objx$ of \cref{minobj} before solving it.

\subsection{Computable bounds on the operator norm} 
When no prior bound on $\opnorm{\xsol}$ is available,
we can invoke \cref{lemma: boundonopnorm} in the following
to estimate $\opnorm{\xsol}$ using any feasible point of \cref{minfeas}. However, 
to obtain a good estimate, we might need to first solve \cref{minfeas}. 

\begin{lemma}\label{lemma: boundonopnorm}
	Suppose \eqref{p} and \eqref{d} admit solutions
	and satisfy strong duality. 
	Let $S$ be feasible for \cref{minfeas}.
	Define $\epsilon$, $T$ as in \cref{thm: computablebound} and $\kappa_V=	\frac{\sigma_{\max}(\mathcal{A})}{\sigma_{\min}(\mathcal{A}_V)}$.
	Define the scaled distance bound $\phi = (1+\kappa_V)\sqrt{\frac{\epsilon}{T}}$
	and the infeasibility $\delta_S = \twonorm{\mathcal{A}_V(S)-b}$. If $\sigma_{\min}
	(\mathcal{A}_V)>0$ $T>0$.
	Then $\opnorm{\xsol }\leq B$ for some constant $B$,
	where
	\vspace{-10pt}	
	\begin{align}\label{boundofXstar}
		B  = \constantB.
		\vspace{-10pt}
	\end{align}
\end{lemma}
\vspace{-5pt}
\begin{proof}
	Use inequality \cref{eq: lemmaminfeas1e1} in \cref{lemma:	minfeas1} to see
	$\fronorm{S-\ssol}\leq \frac{\delta_S}{\sigma_{\min}(\mathcal{A}_V)}$
	for a minimizer $\ssol$ of \cref{minfeas}. Combine this with \cref{fineq} in \cref{lemma:
		minfeas1} to obtain
	\vspace{-10pt}	
	\[
	\fronorm{VS V^\top  - \xsol } \leq (1+ \kappa_V)\bigg( \frac{\epsilon}{T} +
	\sqrt{2\frac{\epsilon}{T} \opnorm{\xsol}}\bigg) +
	\frac{\delta_S}{\sigma_{\min}(\mathcal{A}_V)}.
	\vspace{-5pt}
	\]
	Because $\opnorm{VSV^\top -\xsol} \geq \opnorm{X}- \opnorm{S}$, we further have
	\vspace{-10pt}
	\[
	\opnorm{\xsol}- \opnorm{S} \leq (1+ \kappa_V)\bigg( \frac{\epsilon}{T} +
	\sqrt{2\frac{\epsilon}{T} \opnorm{\xsol}}\bigg) +
	\frac{\delta_S}{\sigma_{\min}(\mathcal{A}_V)}.
	\vspace{-5pt}
	\]
	Solve the above inequality	for	$\opnorm{\xsol}$ to find a formula for the bound $B$.
\end{proof}

The quantities $T,V$ appearing in \cref{thm: computablebound}
can all be computed from available information. The problem is evaluating 
$\sigma_{\max}(\Amap)$ appearing in $\kappa_V$, which 
might require evaluating $\Amap$ on full $n\times n$ matrices. Of course, 
It might be possible to know $\sigma_{\max}(\Amap)$ in priori if we have some 
structure information of it. 

\subsection{A few words on well-posedness} \label{par:wellposedness}
\Cref{thm: computablebound} makes no guarantee on the quality
of the reconstructed primal when
$\min\{\sigma(\mathcal{A}_V), T\} = 0$.
In fact, this failure signals either that $y$
is far from optimality, or that the primal~\cref{p} or dual~\cref{d}
is degenerate (violating the assumptions from \Cref{sec:assumptions}).

To see this, suppose for simplicity,
we know the rank $r=\rsol$ of a solution to~\cref{p}.
If $y$ is close to $\ysol$ and the primal~\cref{p} and dual~\cref{d} are degenerate,
then \cref{lemma: minfeas2} shows that the quantities $\min\{\sigma(\mathcal{A}_V), T\} $
are close to $\min\{\sigma(\mathcal{A}_\vrepox), \lambda_{n-\rsol}(\zsol)\}$.
Furthermore, \cref{lemma: minfeas2} shows that our assumptions
(from \Cref{sec:assumptions}) guarantee
$\min\{\sigma(\mathcal{A}_\vrepox), \lambda_{n-\rsol}(\zsol)\} >0$.
Thus if $\min\{\sigma(\mathcal{A}_V), T\} = 0$, then either
we need a more accurate solution to the dual problem
to recover the primal,
or the problem is degenerate and our assumptions fail to hold.
\section{Lemmas for fixing infeasible dual iterates in \Cref{sec: sd}}

We present one lemma to bound the infeasibility of a dual vector $y$,
and another to show how to construct a feasible $y$ from an infeasible one.

\begin{lemma}\label{lm: boudningdualinfeasibility}
	Suppose \eqref{p} and \eqref{d} admit solutions
	and satisfy strong duality, \Cref{eqn:strong-duality}. Let $\ubar{\alpha}:\,= \inf_{X\in \xsolset}\tr(X)$ where 
	$\xsolset$ is the solution set of \eqref{p}. 
	For any dual vector $y$
	with suboptimality $\tr(C\xsol)- \objp(y) \leq \epsilon$ with $\alpha >
	\ubar{\alpha}$, we have
	$
	\lambda_{\min} (Z(y)) \geq -\dfrac{\epsilon}{\alpha - \ubar{\alpha}}.
	$
\end{lemma}

This lemma shows infeasibility decreases at the same speed as suboptimality.
\begin{proof}
	Let $Z = C - \mathcal{A} ^\top y$. Assume $\lambda_{\min}(Z)<0$.
	(Otherwise, we are done.) 
	Then for any $\xsol \in \xsolset$
	\beq
	\ba{lll}
	\tr (C\xsol) - \objp(y)&  = \tr(C\xsol)  -b^\top y- \alpha \lambda_{\min}(Z)  \\
	&	 = \tr(C\xsol) - (\mathcal{A}\xsol )^{\top} y -\alpha \lambda_{\min}(Z)\\
	&	= \tr(Z\xsol ) -\alpha \lambda_{\min} (Z).
	\ea
	\eeq
	Using the suboptimality assumption and Von Neumann's inequality, we further
	have
	\beq
	\ba{l}
	\epsilon \geq \tr(Z\xsol) -\alpha \lambda_{\min}(Z)\geq \tr(\xsol)
	\lambda_{\min}(Z) -\alpha \lambda_{\min}(Z) .
	\ea
	\eeq
	Rearrange to see
	$\lambda_{\min} (Z) \geq -\dfrac{\epsilon}{\alpha - \tr(\xsol)}.$
	Let	$\tr(\xsol) \to \ubar{\alpha}$ to obtain the result.
\end{proof}

We next show how to construct an  $\epsilon$-suboptimal and feasible dual
vector
from an $\epsilon$-suboptimal and potentially infeasible dual vector.

\begin{lemma}\label{lm: averaging}
	Suppose \eqref{p} and \eqref{d} admit solutions
	and satisfy strong duality, \Cref{eqn:strong-duality}.
	Further suppose a dual vector $y_1$ with $Z_1 = C-\mathcal{A}^\top y_1$
	is infeasible with
	$-\epsilon\leq \lambda_{\min} (Z_1)<0$ and $y_2$ with $Z_2=C-\mathcal{A}^\top y_2$
	is strictly feasible in the sense that $\lambda_{\min}(Z_2)>0$, then the dual
	vector $$y_\gamma = \gamma y_1 + (1-\gamma)y_2$$
	is feasible for $\gamma= \dfrac{\lambda_{\min}(Z_2)}{\epsilon +
		\lambda_{\min}(Z_2)}$. The objective value of
	$y_\gamma$ is
	$$\objp(y_\gamma) =  \dfrac{\lambda_{\min}(Z_2)}{\epsilon +
		\lambda_{\min}(Z_2)}  b^\top y_1  +   \dfrac{\epsilon}{\epsilon +
		\lambda_{\min}(Z_2)} b^\top y_2$$
\end{lemma}
\begin{proof}
	The results follow from the linearity of $C-\mathcal{A}^\top y$ and the
	concavity of $\lambda_{\min}$.
\end{proof}

\section{Additional numerics}\label{sec: additionalNumerics}
\subsection{Primal Recovery for Matrix Completion}\label{sec: primalrecoveryMatrixCompletion} For matrix completion, we generate a random rank $5$ matrix $\trux \in
\reals^{1000\times 1500}$.
We generate the set $\Omega$ by observing each entry
of $\trux$ with probability $0.025$ independently.
To evaluate our method, we compare the recovered primal with $\trux$, which
(with high probability) solves the Matrix Completion problem
\cite{candes2010power}. The rest of the setting is exactly the same as in \cref{sec: primalRecovery}.
The plot is shown in  and we come to the same conclusion as in \cref{sec: primalRecovery}. 
\begin{figure}[tbhp]
	\centering
	\includegraphics[width=1\linewidth]{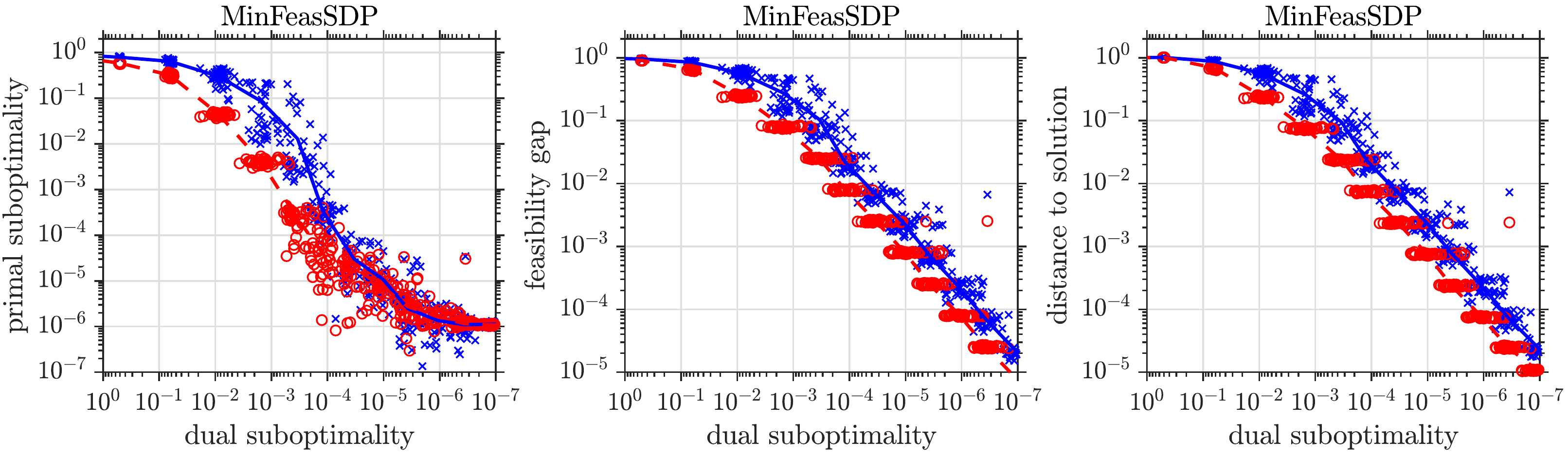}
	\caption{\cref{minfeas}}
	\includegraphics[width=1\linewidth]{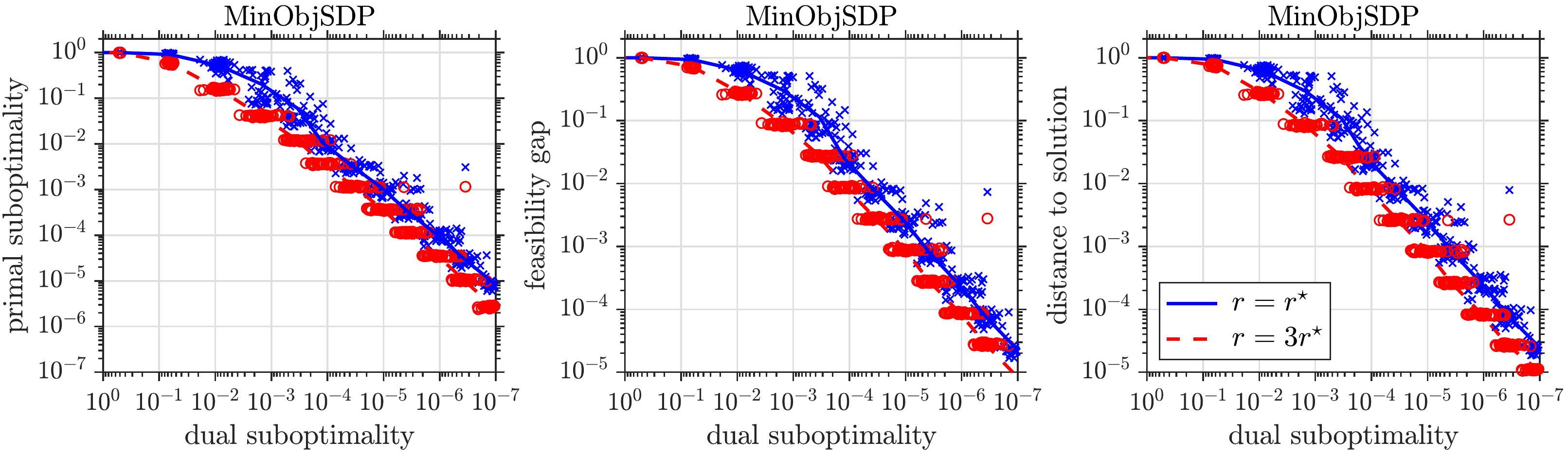}
	\caption{\cref{minobj}}
	\caption[Matrix Completion]{The plots shows the primal recovery performance of
		\cref{minfeas} (upper) and \cref{minobj} (lower) in terms of primal
		suboptimality,
		infeasibility and the distance to solution for the matrix completion problem. The horizontal axis is the dual
		suboptimality. The blue dots corresponds to the choice $r=\rsol$ and the red
		dots corresponds to the choice  $r=3\rsol$ in \cref{alg: accelgrad+minfeas}.}\label{fig:proofofconceptMatrixCompeltion}
\end{figure}
\subsection{Solving primal SDP with various dual solvers of 
	medium scale problems}
In this section we show that \Cref{alg:	accelgrad+minfeas}
(Dual Algorithm $+$ Primal Recovery) solves the primal SDP,
using the dual solvers
AdaGrad \cite{duchi2011adaptive}, AdaNGD \cite{levy2017online},
and AcceleGrad. Here we perform the primal 
recovery in every iteration of the dual algorithms. The problem 
instance for max-cut is the same as \Cref{sec: primalRecovery},
and the instance for matrix completion is the same as \Cref{sec: primalrecoveryMatrixCompletion}.
Here we use \cref{minfeas} to recover the primal.
The numerical results are shown in \Cref{fig:Max-CutMc}.
We plot the relative dual suboptimality, primal suboptimality, infeasibility
and distance to solution (as explained in \cref{sec: primalRecovery}) for each iteration of the dual method.
The solid lines show recovery with $r=\rsol$ while
the dotted lines use the higher rank $r=3\rsol$.

We observe convergence in each of these metrics, as expected from theory.
Primal and dual suboptimality converge faster than the other two
quantities, as in \Cref{fig:proofofconcept}.
Interestingly, while
AccelGrad converges much faster than the other algorithms on the dual side,
its advantage on the primal side is more modest.
We again see that the primal recovered using the larger rank $r=3\rsol$
converges more quickly, though interestingly using the higher rank confers
less of an advantage in reducing distance to the solution
than in reducing primal suboptimality and infeasibility.
\begin{figure}[tbhp]
	\centering
	\includegraphics[width=1\linewidth]{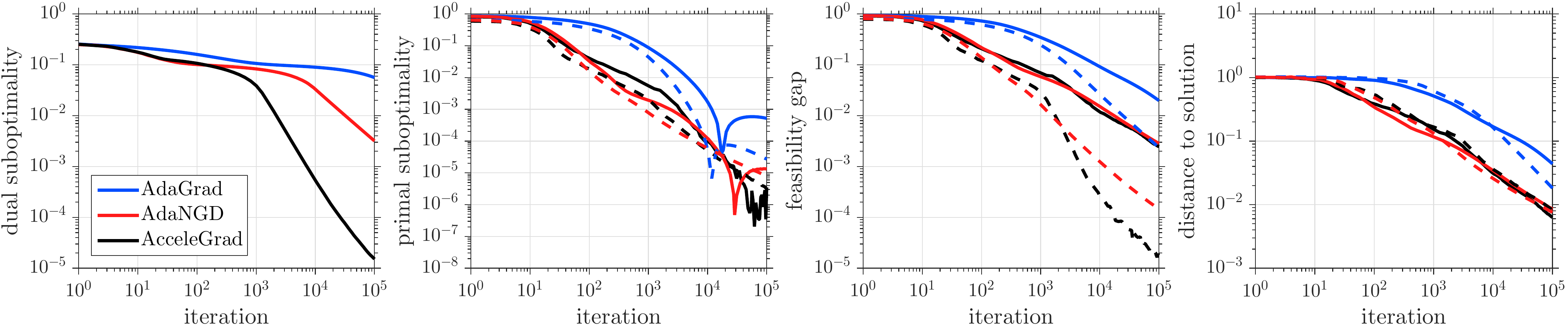}
	\caption{\textbf{Max-Cut}}
	\label{Figure: Max-CutTest1}
	\includegraphics[width=1\linewidth]{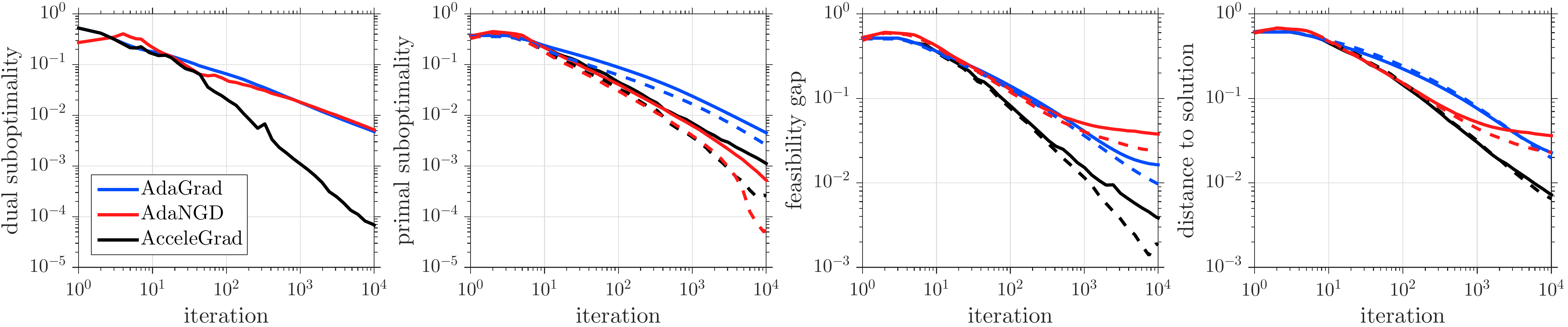}
	\caption{\textbf{Matrix Completion}}
	\label{Figure: MatcompTest1}
	\caption[Max-Cut and Matrix Completion]{Plots from left to right
		columns show convergence of penalized dual objective
		$\objp$, primal suboptimality, infeasibility, and distance to solution.
		The solid lines show recovery with $r=\rsol$ while
		the dotted lines use the higher rank $r=3\rsol$.}\label{fig:Max-CutMc}
	\vspace{-30pt}
\end{figure}
\subsection{Accuracy versus time and comparison to existing solvers} In this section, we present additional numerics regarding accuracy versus time to \Cref{sec: accuracyversusTime}. We perform the same procedure as there to Max-Cut SDP  for G1 ($800^2$), G45 ($10^3\times 10^3$) and G67 ($10^4\times 10^4$) in the Gset, and G\_n\_pin\_pout $(10^5\times 10^5)$ in the DIAMCS10 group. The results can be found in \Cref{fig:accuracy_versus time_and_iteration_counterSupplementMaxCut}.
Results for a matrix completion problem, simulated as described in \Cref{sec: storageVSdimension} with $c=1,10,$ and $100$,
with decision variable size $(n_1+n_2)^2=(125c)^2$ with $n_1=75c$ and $n_2 = 50c$, and $25(n_1+n_2)\log(n_1+n_2)$ 
are shown in \cref{fig:accuracy_versus time_and_iteration_counterSupplementMatrixCompletion}. We also compare with existing solver SDPNAL+ for medium scale problems: Max-Cut problem G45 and matrix completion with $n_1+n_2=1250$. 
We note our method achieves medium accuracy $10^{-3}$ in less than $100$ seconds 
for medium scale problems. Such results are comparable or even better than SDPNAL+. 
\begin{figure}[H]
	\centering
	\includegraphics[width=1\linewidth]{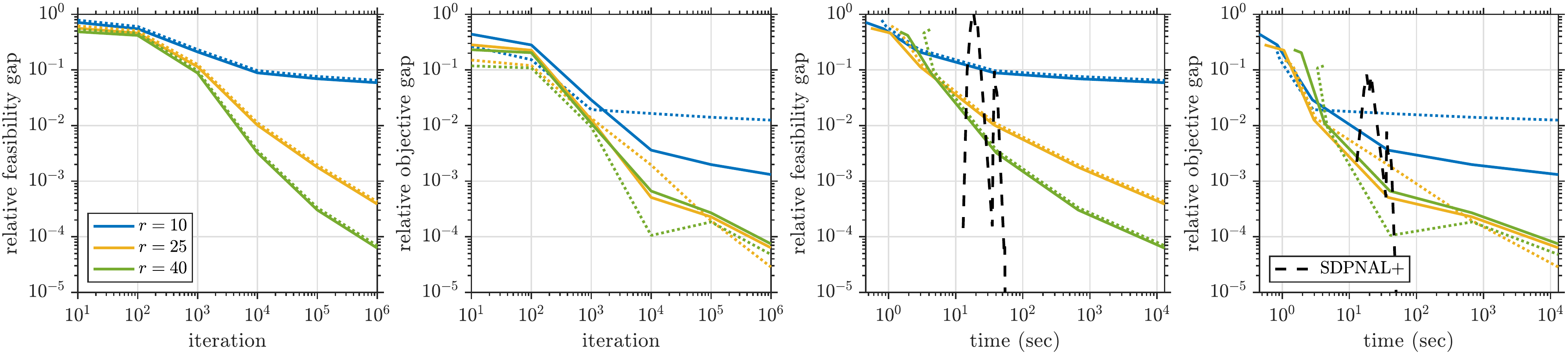}
	\caption{Max-Cut: G1}
	\label{Figure:MaxCutTimeAccuraycG1}
	\centering
	\includegraphics[width=1\linewidth]{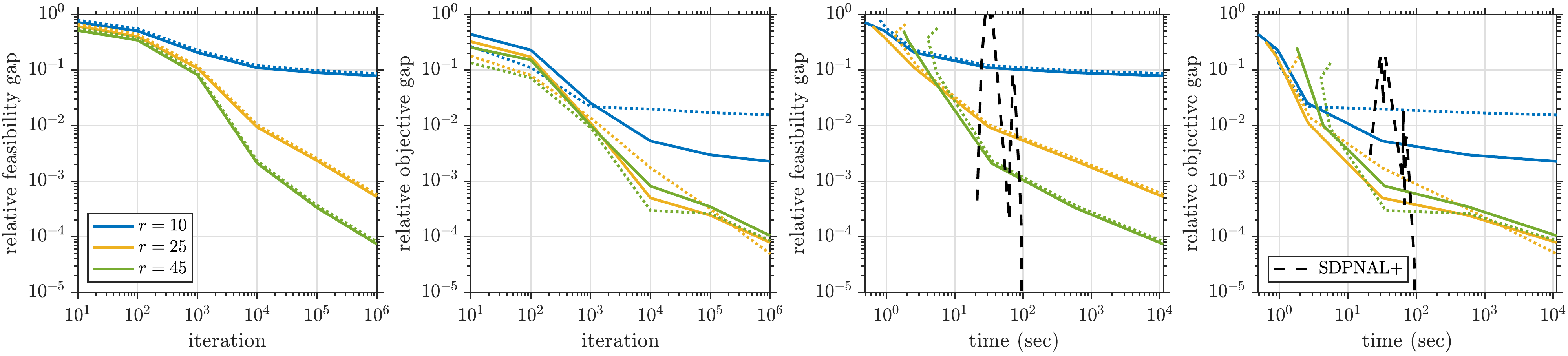}
	\caption{Max-Cut: G45}
	\label{Figure:MaxCutTimeAccuraycG45}
	\centering
	\includegraphics[width=1\linewidth]{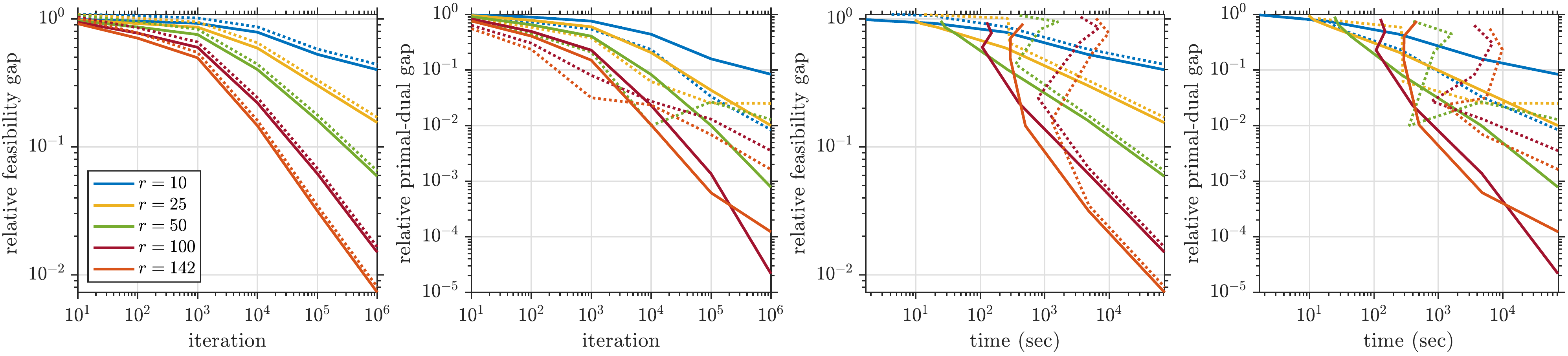}
	\caption{Max-Cut: G75}
	\label{Figure:MaxCutTimeAccuracyG65}
	\centering
	\includegraphics[width=1\linewidth]{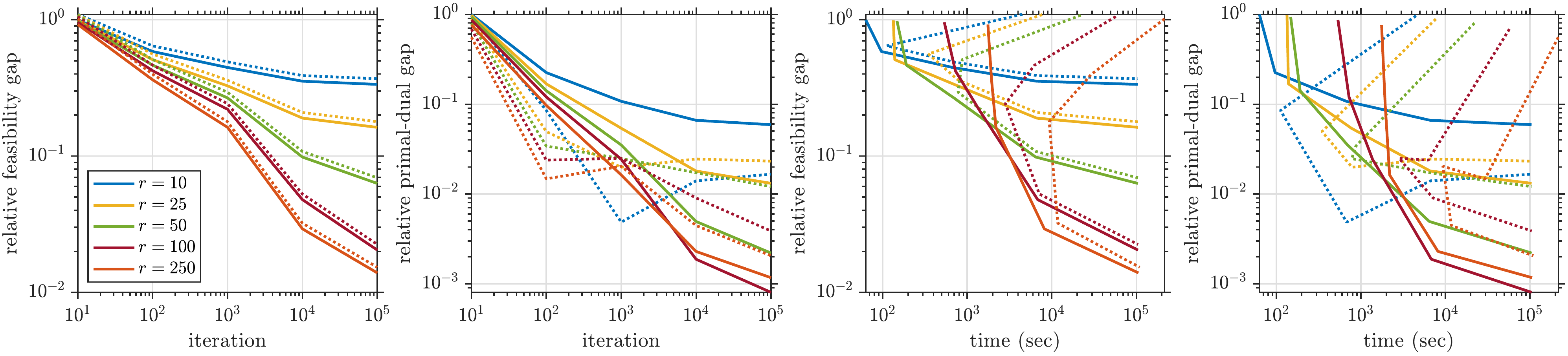}
	\caption{Max-Cut: G\_n\_pin\_pout}
	\label{Figure:MaxCutTimeAccuracyG_n_pin_pout}
	\caption[Accuracy versus actual time and iterations]{Here we show the convergence of scaled suboptimality and infeasibility of our \cref{alg: accelgrad+minfeas} (option 1 as the solid line and option 2 as the dotted line.) against the actual time and iteration counter.}\label{fig:accuracy_versus time_and_iteration_counterSupplementMaxCut}
\end{figure} 
\begin{figure}[H]
	\centering
	\includegraphics[width=1\linewidth]{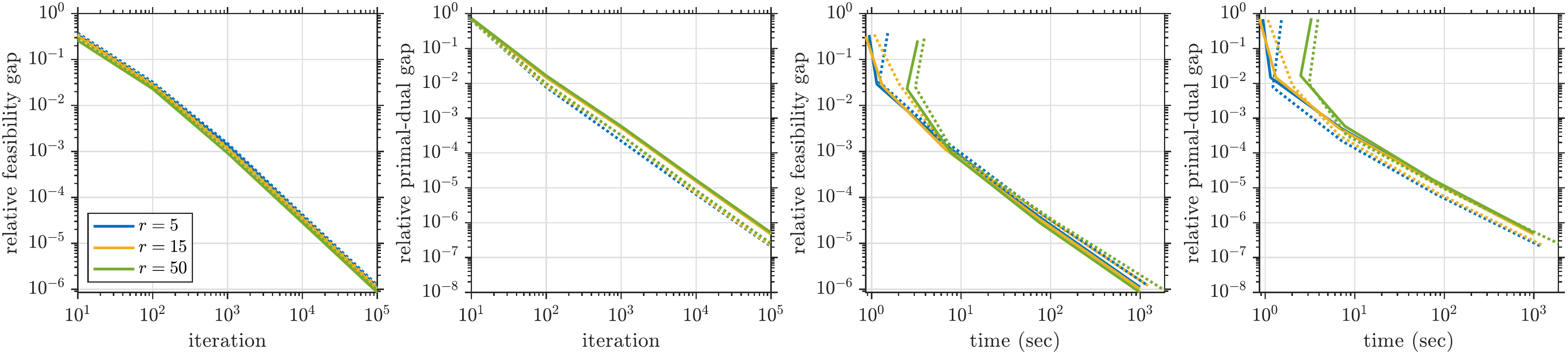}
	\caption{$n_1=75$, $n_2=50$}
	\label{Figure:MatrixCompletion-1}
	\centering
	\includegraphics[width=1\linewidth]{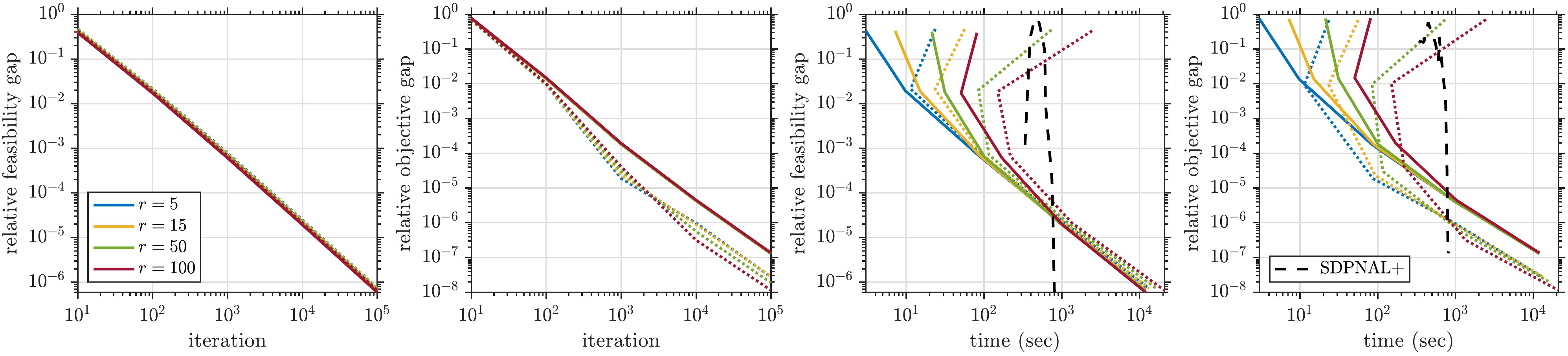}
	\caption{$n_1=750,n_2=500$}
	\label{Figure:MatrixCompletion-10}
	\centering
	\includegraphics[width=1\linewidth]{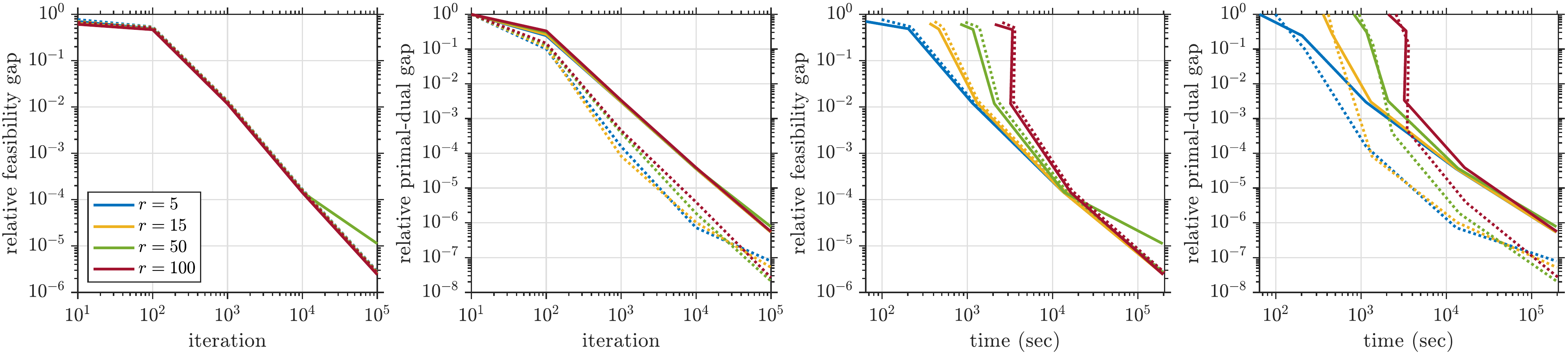}
	\caption{$n_1=7500,n_2=5000$}
	\label{Figure:MaxtrixCompletion-100}
	\caption[Accuracy versus actual time and iterations]{Here we show the convergence of scaled suboptimality and infeasibility of our \cref{alg: accelgrad+minfeas} (option 1 as the solid line and option 2 as the dotted line.) against the actual time and iteration counter.}\label{fig:accuracy_versus time_and_iteration_counterSupplementMatrixCompletion}
\end{figure} 
\section*{Acknowledgments}
Lijun Ding and Madeleine Udell were supported in part by DARPA Award FA8750-17-2-0101. Parts of this research were conducted while Madeleine Udell was in residence at the Simons Institute. Alp Yurtsever and Volkan Cevher have received funding for this project from the European Research Council (ERC) under the European Union's Horizon $2020$ research and innovation programme (grant agreement no~$725594$-time-data), and from the Swiss National Science Foundation (SNSF) under grant number $200021\_178865/1$. Joel A.~Tropp was supported in part by ONR Awards No. N-00014-11-1002, N-00014-17-12146, and N-00014-18-12363.
	\bibliographystyle{siamplain}
	\bibliography{references}
\end{document}